\documentclass[11pt,reqno]{amsart}
\usepackage{fullpage}
\usepackage[margin=1in]{geometry}
\usepackage{verbatim}
\usepackage[foot]{amsaddr}
\usepackage{lineno}
\usepackage{amsmath}
\usepackage{bbm}
\usepackage{graphicx}
\usepackage{float}
\usepackage{mathrsfs}
\usepackage[english]{babel}
\usepackage[utf8]{inputenc}
\usepackage{tikz-cd}
\usepackage{adjustbox}
\usepackage{amsthm}
\usepackage{amssymb}
\usepackage{listings}
\usepackage{xcolor}
\usepackage{enumerate}
\usepackage{thmtools}
\usepackage{thm-restate}
\usepackage[ruled,vlined]{algorithm2e}
\usepackage{epsfig}
\usepackage{subfigure}
\providecommand{\U}[1]{\protect\rule{.1in}{.1in}}
\usepackage[all]{xy}
\usepackage{caption}
\usepackage{longtable}
\usepackage{mathtools}

\usepackage{url}
\usepackage[pdftex,pdfborder={0 0 0},
colorlinks=true,
linkcolor=blue,
citecolor=red,
pagebackref=false,
]{hyperref}
\usepackage{tcolorbox}
\newcommand{\R}{\mathbb{R}}

\DeclareMathOperator*{\argmin}{argmin}
\DeclareMathOperator*{\argmax}{argmax}

\newtheorem{thm}{Theorem}[section]
\newtheorem{prop}[thm]{Proposition}
\newtheorem{lem}[thm]{Lemma}
\newtheorem{cor}[thm]{Corollary}
\theoremstyle{definition}
\newtheorem{ass}{Assumption}

\newtheorem{ass2}{Assumption}
\newtheorem{defn}[thm]{Definition}
\theoremstyle{remark}
\newtheorem{rem}[thm]{Remark}

\numberwithin{equation}{section}

\title{Error estimates of a theta-scheme for second-order mean field games}

\date{\today}

\author{J.\@ Frédéric Bonnans$^1$, Kang Liu$^{1,2}$ and Laurent Pfeiffer$^1$}
\address{$^1$Universit{\'e} Paris-Saclay, CNRS, CentraleSup{\'e}lec, Inria, Laboratoire des signaux et syst{\`e}mes, 91190, Gif-sur-Yvette, France.}
\address{$^2$Institut Polytechnique de Paris, CNRS, Ecole Polytechnique, CMAP, 91120 Palaiseau, France.}
\email{frederic.bonnans@inria.fr, kang.liu@polytechnique.edu, laurent.pfeiffer@inria.fr}

\begin{document}


\begin{abstract}  We introduce and analyze a new finite-difference scheme, relying on the theta-method, for solving monotone second-order mean field games. These games consist of a coupled system of the Fokker-Planck and the Hamilton-Jacobi-Bellman equation. The theta-method is used for discretizing the diffusion terms: we approximate them with a convex combination of an implicit and an explicit term. On contrast, we use an explicit centered scheme for the first-order terms. Assuming that the running cost is strongly convex and regular, we first prove the monotonicity and the stability of our theta-scheme, under a CFL condition. Taking advantage of the regularity of the solution of the continuous problem, we estimate the consistency error of the theta-scheme. Our main result is a convergence rate of order $\mathcal{O}(h^r)$ for the theta-scheme, where $h$ is the step length of the space variable and $r \in (0,1)$ is related to the H\"older continuity of the solution  of the continuous problem and some of its derivatives. \end{abstract}

\maketitle

\section{Introduction}

Mean field games (MFGs), introduced in 2006 independently by J.-M.\@ Lasry and P.-L.\@ Lions in \cite{lasry2007mean} and M.\@ Huang et al. in \cite{huang2006large}, describe the asymptotic behavior of Nash equilibria in stochastic differential games, as the number of players goes to infinity. In this type of games, players have symmetric dynamics and payoff function. The latter function depends on the own strategy of a given player and on an interaction cost depending on the distribution of all players. Mean field games have important applications in various domains, like crowd motion \cite{lachapelle2010computation}, sociology, biology, macroeconomics \cite{achdou2014partial}, trade crowding \cite{cardaliaguet2018mean}, and finance.

Second-order MFGs (see \cite{lasry2007mean,achdou2010mean,cardaliaguet2010notes}) are coupled systems, including a backward Hamilton-Jacobi-Bellman (HJB) equation and a forward Fokker-Planck (FP) equation. The source term of the HJB equation depends on the solution $m$ of the FP equation while the velocity (the optimal control) $v$ in the transport term of the FP equation depends on the solution $u$ of the HJB equation. 
Under appropriate hypotheses, we can express 
$v$ as a function of $\nabla u$ at each time.
Let $\mathbb{T}^d$ be the $d$-dimensional torus and let $Q =[0,1]\times \mathbb{T}^d $.
We consider the following second-order MFG: 
\begin{equation}\label{eq:mfg}\tag{MFG}
\left\{\begin{array}{cll}
\mathrm{(i)} \ & -\partial_{t} u-\sigma \Delta u+H^c\left(t, x, \nabla u(t, x)\right) = f^c(t,x , m(t)) \quad & (t,x) \in Q, \\
\mathrm{(ii)} \ & v( t,x )=-H^c_{p}\left( t, x, \nabla u(t, x)\right) & (t, x) \in Q, \\
\mathrm{(iii)} \ & \partial_{t} m-\sigma \Delta m+\operatorname{div}(v m)=0 & (t, x) \in Q, \\
\ \mathrm{(iv)} \ \ & m( 0,x)=m^c_{0}(x), \quad u(1,x)=g^c(x) & x \in \mathbb{T}^d.
\end{array}\right.
\end{equation}
The Hamiltonian $H^c$ is related to the Fenchel conjugate of a running cost $\ell^c$:
\begin{equation}\label{eq:Hc}
H^c(t,x,p) = \sup_{v\in \mathbb{R}^d} \langle -p, v\rangle - \ell^c(t,x,v).
\end{equation}
 
We introduce in this article a theta-scheme for the discretization of \eqref{eq:mfg}; our main result states that, under suitable assumptions, the solution of the theta-scheme converges to the unique solution of \eqref{eq:mfg}. To the best of the authors' knowledge, this article is the first one, in the context of MFGs, to give a precise convergence order for a fully discrete numerical scheme, namely $\mathcal{O}(h^r)$, where $h$ is the step size of the space variable and $r \in (0,1)$ is related to regularity properties of the solution of \eqref{eq:mfg}.

Let us describe more in detail the theta-scheme which we propose.
Let us denote by $\nabla_h$, $\text{div}_h$ and $\Delta_h$ the discrete gradient, divergence and Laplace operators of the centered finite-difference scheme (precise definitions are in Section 2). Let $\theta \in [0,1]$. At any time $t$, the theta-scheme of the FP equation consists of two steps:
\begin{enumerate}
\item An explicit scheme for an intermediate FP equation, with a weight $(1-\theta)$ for the Laplacian term:
\begin{equation}\label{eq:S1} \tag{S1}
\frac{m(t+1/2)-m(t)}{\Delta t} - (1-\theta) \sigma \Delta_h m(t) + \text{div}_h( mv(t) )= 0.
\end{equation}
\item An implicit scheme for an intermediate heat equation (without divergence term):
\begin{equation}\label{eq:S2} \tag{S2}
\frac{m(t+1) - m(t+1/2)}{\Delta t} - \theta \sigma \Delta_h(m(t+1)) =0.
\end{equation}
\end{enumerate}
Notice that when there is no divergence term ($v=0$), the above scheme \eqref{eq:S1}-\eqref{eq:S2} coincides with the classical theta-scheme for the heat equation \cite{allaire2007numerical}.
For the HJB equation, we propose an adjoint scheme; at each time $t$, two steps are performed: (1) an implicit scheme for an intermediate heat equation (without the Hamiltonian term) and (2) an explicit scheme for an intermediate HJB equation. The adjoint structure of the coupled system \eqref{eq:mfg} is preserved in the resulting discretized system, which is an important property for the analysis.

\textit{Motivations of the theta-scheme}. Let us describe the main properties of the theta-scheme, which justify our interest for it. If $\theta=0$, our scheme is an explicit scheme which has a natural interpretation as a discrete mean field game. However, it is not clear whether the explicit scheme for the FP equation, when $\theta=0$, enjoys stability properties for some $\ell^2$-norm.
 To ensure stability, a natural idea consists in taking an implicit scheme for the second-order term, i.e.\@ $\theta = 1$. This yields a mixed scheme (implicit for the Laplacian term and explicit for the divergence term). We emphasize that the divergence term should remain explicit, in order to guarantee that the discrete system has a structure of a discrete MFG. When $\theta=1$, we see that \eqref{eq:S1} is an explicit scheme of a continuity equation (without diffusion term). To ensure the monotonicity of \eqref{eq:S1}, an upwind discretization for the divergence term should be employed, instead of centered scheme.
 In comparison with a centered discretization, the upwind discretization has the following disadvantages: (1) the consistency error is of a lower order, (2) we need then to construct a numerical Hamiltonian (see \cite{achdou2010mean,achdou2013mean}) to preserve the adjoint structure. Finally, we propose to take $\theta \in (1/2,1)$ in  \eqref{eq:S1}-\eqref{eq:S2} and to keep the centered scheme for the first-order term. The $\ell^2$-stability is proved in Proposition \ref{prop:energy} for the case when $\theta > 1/2$. The monotonicity property is obtained under a CFL condition \eqref{cond:CFL3}, for all $\theta<1$, see Theorem \ref{thm:equivalence}.
 We end up with a discrete system which has a structure of a discrete MFG, has a higher order for the consistency error, and which does not require the construction of a numerical Hamiltonian.

Under suitable assumptions, MFGs have a potential structure (see \cite[Def.\@ 1.1]{cardaliaguet2017learning}), i.e.\@ the system \eqref{eq:mfg} can be interpreted as the first order optimality condition of an optimal control problem of the FP equation, see \cite{lasry2007mean,lachapelle2010computation,lavigne2021}. Then some optimization algorithms can be applied to solve this optimal control problem, such as the fictitious play \cite{cardaliaguet2017learning}, the generalized conditional gradient algorithm \cite{lavigne2021}, ADMM and Chambolle-Pock's algorithm \cite{achdou2020mean,lavigne2023discrete}, etc. The last important feature of our theta-scheme is that it preserves the potential structure (when it exists), which allows the application of the previously mentioned methods directly on the discrete system. 
{These methods avoid solving a large discrete nonlinear forward-backward system. For instance, the fictitious play \cite{cardaliaguet2017learning} and the generalized conditional gradient algorithm \cite{lavigne2021} require to solve the discrete HJB and FP equations iteratively.
One significant difference between the theta-scheme and the implicit scheme proposed in \cite{achdou2013mean} is that the first-order terms in the discrete HJB and FP equations of the former are explicit. 
Thanks to this, at each time step of the discrete HJB equation, the difficulty of our method lies in solving a linear equation associated with the implicit part of the theta-scheme, which is much cheaper than solving a nonlinear algebraic equation in the totally implicit scheme \cite{achdou2013mean}.
We mention that the aforementioned linear equation to be solved is an implicit scheme of a heat equation. Consequently, in high-dimensional cases, we can consider splitting methods \cite[Sec.\@ 4.4]{thomas1995numerical} to decompose the discrete Laplace operator and reduce computational complexity.}

\textit{Related works}. In 2010, a first result concerning the convergence of a finite-difference scheme for stationary MFGs was obtained in \cite{achdou2010mean}. In this paper, the authors also proposed an implicit scheme for time-dependent MFGs and proved the existence and uniqueness of the solution of this scheme. In 2013, a convergence result was obtained for the same implicit scheme in \cite{achdou2013mean} when the Hamiltonian has a monomial form, i.e.\@ $H^c(x,p) = \mathcal{H}(x)+|p|^{\beta}$, with $\beta\in(1,+\infty)$. The two cited works assume the existence of a classical solution for \eqref{eq:mfg}. In 2016, in the absence of this existence assumption, \cite{achdou2016} proved that the solution of the implicit scheme converges to a weak solution of \eqref{eq:mfg} when the grid steps tend to zero.  No assumption on the Hamiltonian is made in \cite{achdou2016}, but a technical assumption, Assumption (g5), is required for the numerical Hamiltonian (the discrete counterpart of the Hamiltonian). An example of a numerical Hamiltonian satisfying (g5) is only presented for a Hamiltonian with a monomial form (as above), with $\beta\in (1,2]$.

Other discretization techniques have been considered in the literature.
{We mention the articles \cite{carlini2014fully,carlini2015semi} in which a semi-Lagrangian discretization is proposed for first-order and  second-order MFGs, respectiveley. The well-posedness of the resulting discrete system is established for both cases. In \cite{carlini2015semi}, the scheme's convergence is proven for non-degenerate second-order MFGs in any dimension and for degenerate second-order and first-order cases in dimension one.}
A sort of semi-Lagrangian discretization is proposed in \cite{hadikhanloo2019finite} for first-order MFGs and convergence is established in general dimension. In \cite{bertucci2022} a semi-discretization in space, with finite differences, is investigated.
It is shown that the solution of the semi-discrete master equation converges to the solution of the continuous master equation, with an explicit rate of convergence.
Finally, we cite the article \cite{achdou2020mean}, which gives a good summary of the numerical methods for MFGs.

\textit{Numerical analysis}.
In this paper, we assume that  the running cost $\ell^c$ is strongly convex with respect to the control variable. This is equivalent to the Lipschitz continuity of $\nabla H^c$ with respect to its third variable. This assumption plays a key role in the stability analysis.
We assume that the coupling function $f^c$ is Lipschitz continuous w.r.t.\@ $x$ and with respect to $m$, for the $\mathbb{L}^2$-norm. Note that our regularity assumptions on $f^c$ are stronger than those of \cite{achdou2013mean}.
We also make a monotonicity assumption for $f^c$, in Lasry and Lions' sense, see \cite[Thm.\@ 2.4]{lasry2007mean}.
This assumption ensures the uniqueness of the solution of \eqref{eq:mfg}.
For the consistency analysis, we assume that the exact solution of \eqref{eq:mfg} lies in the H\"older space $\mathcal{C}^{1+r/2,2+r}(Q)$ (see \cite[Ch.\@ 8.5]{krylov1996lectures} for the definition). 
In Appendix \ref{Appendix:B}, we provide sufficient conditions on the data for this regularity assumption to hold, for an exponent $r$ which is explicit.
We also make use of assumptions dealing with the regularity of $\ell^c$, $m_0^c$ and $g^c$.
Our convergence analysis relies on a consistency analysis and a stability analysis, the latter relies on a fundamental inequality and an energy estimate for the discrete FP equation.

\textit{Consistency analysis.} We prove that the discrete HJB equation has a consistency error of order $\mathcal{O}(\Delta t h^r)$ at each time step. For the discrete FP equation, the consistency error is the sum of two terms: one is in the form of the discrete divergence of a term of order $\mathcal{O}(\Delta th^{2r+d})$ (which can be dealt with by a discrete integration by parts formula in the convergence proof), the other one is of order $\mathcal{O}(\Delta t h^{r+d})$. In comparison with \cite{achdou2010mean,achdou2013mean}, there is no numerical Hamiltonian in our scheme. This simplifies the consistency analysis and avoids the treatment of an additional error term.

\textit{Fundamental inequality.} The fundamental inequality (Proposition \ref{prop:error_fund}) is established for a general class of discrete MFGs, for which the existence and uniqueness of a solution is easily obtained with a standard fixpoint approach.
The fundamental inequality allows us to quantify the variation of the control variable $v$ when a discrete MFG is subject to perturbations.
It is deduced from equality \eqref{eq:fund10}, which is similar to the fundamental equality proved 
in \cite[Eq.\@ 3.20]{achdou2013mean} for an implicit scheme. Our proof of the fundamental inequality also relies on the following technical lemma, given in \cite[Thm.\@ 2.1.5]{nesterov2018lectures}: If $F$ is a convex function with $L$-Lipschitz gradient, then for any $p, q$, it follows that
\begin{equation}\label{eq:nesterov}
\frac{1}{2L} \|\nabla F(p) - \nabla F(q)\|^2 \leq F(p) - F(q) - \langle \nabla F(q), p-q\rangle .
\end{equation}
We give a second proof of the fundamental inequality, which does not rely on the fundamental equality \eqref{eq:fund10}. Instead we define a ``relative'' potential function, and deduce the fundamental inequality from upper and lower bounds of this relative potential function.

\textit{Energy estimate.} We provide in Proposition \ref{prop:energy} an upper bound of the $\ell^2$-norm of the solution of the discrete FP equation under some perturbations. The proof of the energy inequality is inspired by the one for parabolic PDEs, see \cite{Liions1971,ladyvzenskaja1988linear}, and the one for the implicit scheme, see \cite{achdou2013mean}.

\textit{Numerical Hamiltonian}.
As we mentioned earlier, it is assumed in \cite{achdou2016} that the numerical Hamiltonian satisfies a specific assumption, Assumption (g5). It turns out that when the numerical Hamiltonian is convex and has a Lipschitz gradient, then (g5) can be easily deduced from inequality \eqref{eq:nesterov}, as we show in Lemma \ref{lm:numericalH}.
Using this technical result, we provide an example of a numerical Hamiltonian which satisfies all the assumptions of \cite{achdou2016}, for the case of a running cost which is strongly convex with respect to the control variable, uniformly in time and space. See Theorem \ref{thm:numericalH}. This result is of independent interest since our theta-scheme does not require the construction of a numerical Hamiltonian.

\textit{Organization of the paper}. In Section 2, we present the theta-scheme and state our main result. Section 3 is dedicated to a general class of discrete MFGs (covering the theta-scheme). We prove in this section the fundamental inequality. In Section 4, some properties of the theta-scheme are demonstrated, in particular, we prove the announced energy estimate for the FP equation. The consistency analysis and the proof of the main result are given in Section 5.

\section{The theta-scheme and the convergence result}

\subsection{Preliminaries}

The set of functions from some finite set $A$ to $\mathbb{R}$ (resp.\@ $\mathbb{R}^d$) is denoted by $\mathbb{R}(A)$ (resp.\@ $\mathbb{R}^d(A)$):
\begin{equation*}
   \mathbb{R}(A) = \{ m \colon A \rightarrow \mathbb{R} \}, \qquad \mathbb{R}^d(A) = \{ m \colon A \rightarrow \mathbb{R}^d \}.
\end{equation*}
{
Let us introduce the set of probability measures on $A$, defined by
\begin{equation*}
    \mathcal{P}(A)  = \Big\{m\in \mathbb{R}(A) \, \Big| \, \forall x\in A, \, m(x)\geq 0, \, \sum_{y\in A} m(y) =1 \Big\}.
\end{equation*}}
We denote by $\|\cdot\|$ and $\langle \cdot , \cdot
\rangle$ the Euclidean norm and the scalar product in $\mathbb{R}^n$.
We define below a scalar product and a norm for functions defined on a finite set.

\begin{defn}
Let $n \in \mathbb{N}_{+}$ and let $A_1$ and $A_2$ be two finite sets. For any $\mu, \nu\in \mathbb{R}^n(A_1)$ and $p\in[1,\infty)$, we define
\begin{equation*}
 \langle \mu, \nu \rangle = \sum_{x\in A_1} \langle \mu(x), \nu(x) \rangle ; \qquad  \|\mu\|_p = \Big(\sum_{x\in A_1} \|\mu(x)\|^p \Big)^{1/p} ;\qquad\|\mu\|_{\infty} = \max_{x\in A_1} \|\mu(x)\|.
\end{equation*}
For any $\mu \in \mathbb{R}^n(A_1\times A_2)$ and $p_1, p_2 \in [1,\infty]$, we define
\begin{equation*}
\|\mu\|_{p_1,p_2} = \Big\| \big( \|\mu(x,\cdot)\|_{p_2} \big)_{x \in A_1} \Big\|_{p_1} = \begin{cases}
        \big(\sum_{x\in A_1} \|\mu(x,\cdot)\|_{p_2}^{p_1} \big)^{1/p_1}, \quad &\text{if } p_1\in [1,\infty),\\[0.5em]
        \max_{x\in A_1} \|\mu(x,\cdot)\|_{p_2}, & \text{if } p_1=\infty.
\end{cases}
\end{equation*}
\end{defn} 

\begin{lem}[H\"{o}lder's inequality]
Let $\mu, \nu  \in \mathbb{R}^n(A_1\times A_2)$. Then,
\begin{equation*}
    \sum_{x_1\in A_1} \sum_{x_2\in A_2}\Big|\langle \mu(x_1,x_2) , \nu(x_1,x_2) \rangle \Big| \leq \|\mu\|_{p_1,p_2} \|\nu\|_{p_1^{*},p_2^{*}},
\end{equation*}
 where $p_i \in [1,\infty]$ and $1/p_i + 1/ p_i^{*} = 1$,  for $i=1,2$.
\end{lem}

We make use of Nemytskii operators, in order to alleviate some notations.

\begin{defn}[Nemytskii operators]
Let $\zeta \colon \mathcal{X}\times \mathcal{Y} \rightarrow \mathcal{Z}$ and let $u\colon \mathcal{X} \rightarrow \mathcal{Y}$. Then, the associated Nemytskii operator is the mapping $\zeta[u]$, defined from $\mathcal{X}$ to $\mathcal{Z}$ by
\begin{equation*}
    \zeta[u](x) = \zeta(x,u(x)).
\end{equation*}
\end{defn}

\subsection{Notations for the finite-difference scheme}

The time step is $\Delta t= 1/T$,  for $T\in \mathbb{N}_{+}$. We assume that $T>1$. The set of time indices is denoted by $\mathcal{T}$ ($\tilde{\mathcal{T}}$ when the final time $T$ is included):
\begin{equation}\label{T}
\mathcal{T} = \{0,1,\ldots,T-1\} ; \qquad \mathcal{\tilde{T}} = \{0,1,\ldots, T\}.
\end{equation}
Let $S$ be the uniform discretization of the torus $\mathbb{T}^d$ with step size $h = 1/N$, for $N\in \mathbb{N}_{+}$, defined by
\begin{equation}\label{S}
    S = \big\{ (i_1,i_2,\ldots,i_d) h \; \mid \;  i_1,\ldots,i_d \in \mathbb{Z}/N\mathbb{Z} \big\}.
\end{equation}
 Let $(e_i)_{i=1,\ldots,d}$ be the natural canonical basis of $\mathbb{R}^d$. The discrete Laplace, gradient, and divergence operators for the centered finite-difference scheme are defined as follows:
\begin{align*}
    & \Delta_h \mu(x) = \sum_{i=1}^d \frac{  \mu(x+he_i) + \mu(x-he_i)  -2 \mu(x)}{h^2}, & \forall \mu\in \mathbb{R}(S), \; \forall \; x\in S,\\
    & \nabla_h \mu(x) = \Big( \frac{\mu(x+he_i)- \mu(x-he_i)}{2h} \Big)_{i=1}^d, & \forall \mu\in \mathbb{R}(S), \; \forall \; x\in S,\\
    & \text{div}_h \omega (x) = \sum_{i=1}^d \frac{\omega_i(x+he_i) - \omega_i(x-he_i)}{2h}, & \forall \; \omega\in \mathbb{R}^d(S), \; \forall \; x \in S, 
\end{align*}
where $\omega_i$ is the $i^{\text{th}}$ coordinate of $\omega$. The forward discrete gradient is defined by
\begin{equation} \label{eq:forward_grad}
    \nabla_h^{+} \mu(x) = \Big( \frac{\mu(x+he_i)- \mu(x)}{h} \Big)_{i=1}^d, \qquad \forall \mu\in \mathbb{R}(S), \; \forall  x\in S.
\end{equation}

\begin{lem}[Integration by parts formula]\label{lm:int_by_part}
For any $\omega\in \mathbb{R}^d(S)$ and for any $\mu, \nu \in \mathbb{R}(S)$, it holds that
\begin{align} \label{eq:int_by_part1}
         & - \sum_{x\in S} \mu(x) \textnormal{div}_h \omega (x)   = \sum_{x\in S}  \left\langle  \nabla_h {\mu}(x)  , {\omega}(x)\right\rangle ; \\
      \label{eq:int_by_part2}
        & - \sum_{x\in S} \nu (x) \Delta_h \mu (x) = \sum_{x \in S} \left\langle \nabla^{+}_h \nu(x), \nabla^{+}_h \mu(x)  \right\rangle .
\end{align}
\end{lem}

The proof is given in the Appendix \ref{Appendix:A}. 

\begin{lem} \label{lem:easy_ineq}
For any $\mu\in \mathbb{R}(S)$, the following inequality holds:
\begin{equation}\label{eq:dmu+}
 \|\nabla_h \mu\|_2 ^2 \leq  \|\nabla_h^{+} \mu\|_2 ^2 .
\end{equation}
\end{lem}

The proof is given in the Appendix \ref{Appendix:A}. 
The following lemma shows some general properties of the implicit scheme associated with the heat equation $    \frac{\partial m}{\partial t} - c \Delta m = 0$, used in our theta-scheme.

\begin{lem} \label{lm:implicit}
Let $X\in \mathbb{R}^{|S|}$. Consider the scheme
\begin{equation}\label{eq:implicit}
    \frac{Y(x) - X(x)}{\Delta t} - c \Delta_h Y(x) =0 , \qquad \forall \; x \in S,
\end{equation}
with unknown $Y \in \R^{|S|}$. The following holds true.
\begin{enumerate}
    \item (Existence and uniqueness) The scheme \eqref{eq:implicit} has a unique solution $Y$.
    \item (Monotonicity) If $X\geq 0$, then $Y\geq 0$. Moreover, if $X\in\mathcal{P}(S)$, then $ Y\in \mathcal{P}(S)$.
    \item (Lipschitz constant) If $X$ is $L$-Lipschitz, then $Y$ has the same Lipschitz constant $L$.
    \item (Continuity of the discrete gradient and Laplacian) Suppose that $\Delta_h X$ is $\alpha$-H\"older continuous with constant $L'$, where $0< \alpha \leq 1$. Then there exists a constant $C$, independent of $\Delta t$ and $h$, such that
\begin{equation*}
        \|\nabla_h X - \nabla_h Y\|_{\infty} \leq C \Delta t h^{ \alpha -1}, \qquad \|\Delta_h X - \Delta_h Y\|_{\infty} \leq C \Delta t h^{\alpha -2}.
\end{equation*}
\end{enumerate}
\end{lem}

The proof is given in the Appendix \ref{Appendix:A}. 

\subsection{The theta-scheme and the main result}

We describe the MFG system of interest.
Let us fix a running cost $\ell^c$, a coupling cost $f^c$, an initial condition $m_0^c$ and a terminal cost $g$, where
\begin{align*}
\ell^c \colon Q \times \R^d \rightarrow \R,
\qquad
f^c \colon Q \times \mathcal{D} \rightarrow \R,
\qquad
m_0^c \in \mathcal{D},
\qquad
g^c \colon \mathbb{T}^d \rightarrow \R,
\end{align*}
and where the set $\mathcal{D}$ is defined by
$\mathcal{D} =  \big\{ \mu \in \mathbb{L}^2(\mathbb{T}^d) \, \vert \, \mu\geq 0, \int_{\mathbb{T}^d} \mu(x) dx =1  \big\}$.
Recall the formulation of the continuous mean field game:
\begin{equation*}
\left\{\begin{array}{cll}
\mathrm{(i)} \ & -\partial_{t} u-\sigma \Delta u+H^c\left(t, x, \nabla u(x, t)\right) = f^c(t,x , m(t)) & (t,x) \in Q, \\
\mathrm{(ii)} \ & v( t,x )=-H^c_{p}\left( t, x, \nabla u(x, t)\right) & (t, x) \in Q, \\
\mathrm{(iii)} \ & \partial_{t} m-\sigma \Delta m+\operatorname{div}(v m)=0 & (t, x) \in Q, \\
\ \mathrm{(iv)} \ \ & m( 0,x)=m^c_{0}(x), \quad u(1,x)=g^c(x) & x \in \mathbb{T}^d,
\end{array}\right.
\end{equation*}
where $H^c(t,x,p) = \sup_{v\in \mathbb{R}^d} \langle -p, v\rangle - \ell^c(t,x,v)$. 
We make the following assumptions on the data functions.

\begin{ass}\label{ass:continuous}
\begin{enumerate}
\item \emph{Regularity}. The running cost $\ell^c$ is continuously differentiable with respect to $v$. There exist positive constants $L_{\ell}^c$, $L_g^c$, and $L_f^c$ such that for any $(t,x) \in Q$, for any $v \in \R^d$, and for any $m \in \mathcal{D}$,
\begin{itemize}
\item $\ell^c(\cdot,x,v)$, $\ell^c(t,\cdot,v)$, and $\ell_v^c(\cdot,x,v)$ are $L_{\ell}^c$-Lipschitz continuous
\item $g^c$ is $L_g^c$-Lipschitz continuous
\item $f^c(\cdot,x,m)$, $f^c(t,\cdot,m)$, and $f^c(t,x,\cdot)$ are $L_f^c$-Lipschitz continuous (with respect to the $\| \cdot \|_{\mathbb{L}^2}$-norm for the third variable).
\end{itemize}
\item \emph{Strong convexity}. There exists $\alpha > 0$ such that for any $(t,x) \in Q$, $\ell^c(t,x,\cdot)$ is strongly convex with modulus $\alpha^c$, i.e.
\begin{equation*}
 \ell^c(t,x,v_2) \geq \ell^c(t,x,v_1) + \langle \ell_v^c(t,x,v_1) , v_2-v_1\rangle + \frac{\alpha^c}{2}\|v_2-v_1\|^2,
 \quad
 \forall v_1, v_2 \in \R^d.
\end{equation*}
\item \emph{Monotonicity.} The global cost $f^c$ is monotone, i.e., for any $t \in [0,T]$, for any $m_1$ and $m_2 \in \mathcal{D}$,
\begin{equation*}
        \int_{ \mathbb{T}^d} \Big( f^c(t,x',m_1) - f^c(t,x',m_2)\Big) \big(m_1(x') - m_2(x')\big) dx' \geq 0.
\end{equation*}
\end{enumerate}
\end{ass}

\begin{lem} \label{lm:l}
Let Assumption \ref{ass:continuous} hold true. Then  $H^c$ is continuously differentiable with respect to $p$ and $H^c_p$ is $(1/\alpha)$-Lipschitz continuous with respect to $p$.
Moreover, $H^c$ and $H^c_p$ are respectively $L_{\ell}^c$- and $(L_{\ell}^c/\alpha)$-Lipschitz continuous with respect to $t$.
\end{lem}

The proof is given in Appendix \ref{Appendix:A}. 
Following \cite[page 117]{krylov1996lectures}, we introduce the following spaces. Given $r \in (0,1)$, $\mathcal{C}^{r/2, r}(Q)$ denotes the set of real-valued functions over $Q$ which are H\"older continuous with exponent $r$ (resp.\@ $r/2$) with respect to $x$ (resp.\@ $t$). We denote by $\mathcal{C}^{1+r/2, 2+r}(Q)$ the set of real-values functions $Q$ which are such that $m$, $\partial_t m$, $\partial_{x_i} m$, $\partial_{x_i x_j} m$ lie in $\mathcal{C}^{r/2, r}(Q)$, for any $i,j=1,\ldots d$.

We make the following assumption on the solution of \eqref{eq:mfg}.

\begin{ass}\label{ass:sol+}
The continuous mean field game \eqref{eq:mfg} has a unique solution $(u^{*}, v^{*}, m^{*})$, with $u^{*}, m^{*} \in \mathcal{C}^{1+r/2,2+r}(Q) $ and $v^{*}\in \mathcal{C}^r(Q) \cap \mathbb{L}^{\infty}([0,1]; \mathcal{C}^{1+r}(\mathbb{T}^d))$, where $r\in (0,1)$.
\end{ass}

In Appendix \ref{Appendix:B}, we propose a set of regularity assumptions on $\ell^c$, $f^c$, $m_0^c$ and $g^c$ (Assumption \ref{ass:regular}). We show in Theorem \ref{thm:regular} that Assumptions \ref{ass:continuous} and \ref{ass:regular} together imply the Assumption \ref{ass:sol+}, for an explicit value of $r$.

\smallskip

\textbf{Assumptions \ref{ass:continuous} and
\ref{ass:sol+} are supposed to be satisfied throughout the article}.

\smallskip

Let us now discretize the data functions.
Let us define $B_h(x) = \prod_{i=1}^d [x-he_i/2,\, x +he_i/2)$.
We introduce two operators $\mathcal{I}_h \colon \mathbb{R}(\mathbb{T}^d) \rightarrow \mathbb{R}(S)$ and $\mathcal{R}_h \colon \mathbb{R}(S) \rightarrow \mathbb{R}(\mathbb{T}^d)$, defined as follows: For any $m^c\in \mathbb{R}(\mathbb{T}^d)$ and for any $ m \in \mathbb{R}(S)$, 
\begin{equation}
\label{eq:operatorIh}
\begin{array}{rll}
    \mathcal{I}_h (m^c) (x) &\! \! \! \! = {\displaystyle \int_{B_h(x)}} m^c(y)dy,  \quad &\forall 
    x\in S;\\[1.5em]
    \mathcal{R}_h (m) (y) & \! \! \! \! = {\displaystyle \frac{m(x)}{h^d} }, & \forall x\in S, \; y\in B_h(x).
    \end{array}
\end{equation}
The discrete counterparts of the data functions $\ell^c$, $ H^c$, $m_0^c$, and $g^c$ are the functions defined as follows: For any $t\in \tilde{\mathcal{T}}$, $x\in S$ and $p\in \mathbb{R}^d$,
\begin{equation}\label{eq:grid1}
\begin{array}{ll}
   \ell(t,x,p) = \ell^c(t\Delta t, x,p), \qquad &  H (t,x, p) = H^c(t\Delta t,x,p), \\[0.5em]
   m_0(x) = \mathcal{I}_h (m_0^c)(x), \quad & g(x) = g^c(x).
\end{array}
\end{equation}
The discrete counterpart of $f^c$ is the function $f \colon \mathcal{T}\times S\times \mathbb{R}(S)$ to $\mathbb{R}$ defined by
\begin{equation}\label{eq:grid2}
    f(t,x,m) = \frac{1}{h^d}\int_{y\in B_h(x)} f^c\Big(t\Delta t,y,\mathcal{R}_h(m) \Big) dy.
\end{equation}
Taking any $\theta\in [0,1]$, we introduce the theta-scheme of \eqref{eq:mfg}: find $(u,v,m) \in  \R(\bar{\mathcal{T}} \times S) \times \R^d(\mathcal{T} \times S) \times \R(\bar{\mathcal{T}} \times S)$ such that $\forall (t,x)\in \mathcal{T}\times S$,
\begin{equation}\label{eq:theta_mfg2}\tag{$\theta$-MFG}
\left\{\begin{array}{cll}
\mathrm{(i)}  & \begin{cases}
        -\frac{u(t+1,x) - u(t+1/2,x)}{\Delta t} - \theta \sigma \Delta_h u(t+1/2,x) = 0 , \\[0.5em]
         -\frac{u(t+1/2,x) - u(t,x)}{\Delta t} - (1-\theta)\sigma \Delta_h u(t+1/2, x) +  H [\nabla_h u ( \cdot +1/2 , \cdot ) ](t,x)= f(t,x,m(t)); 
\end{cases}\\
~\\
\mathrm{(ii)}  & v(t,x) = -H_p[\nabla_h u(\cdot+1/2,\cdot)](t,x);\\
~\\
\mathrm{(iii)} &
\begin{cases}
        \frac{m(t+1/2.x) - m(t,x)}{\Delta t} -  (1-\theta)\Delta_h m(t,x) + \text{div}_h (vm)(t,x) = 0 ,  \\[0.5em]
            \frac{m(t+1,x) - m(t+1/2,x)}{\Delta t} - \theta \sigma \Delta_h m(t+1,x) = 0;
    \end{cases} \\
~\\
\mathrm{(iv)} \ & m(0,x)=m_{0}(x), \quad u(T,x)=g(x).
\end{array}\right.
\end{equation}
Denoting  $B_1 = \text{Id} - \theta \sigma\Delta t \Delta_h$, the first equation in the dynamic programming equation can be rewritten as follows: $B_1 u(t+1/2,\cdot)= u(t+1,\cdot)$. By Lemma \ref{lm:implicit}, $B_1$ is invertible.
This allows us to consider $u(t+1/2,\cdot)$ as an auxiliary variable, uniquely determined by $u(t+1,\cdot)$, and thus to regard the unknown value function $u$ of the theta-scheme as an element of $\R(\bar{\mathcal{T}}\times S)$. The same argument also holds for the other auxiliary variable $m(t+1/2,\cdot)$.

We fix now a constant $M$, defined as follows:
\begin{equation} \label{eq:cons_M}
M = \frac{1}{\alpha^c}
\Big(
2 \max_{(t,x) \in Q}
\| \ell^c_v(t,x,0) \|
+ \sqrt{d} (L_{\ell}^c + L_f^c + L_g^c)
\Big).
\end{equation}
The constant $M$ is an upper bound of $\| v \|_{\infty,\infty}$, as will be seen in Theorem \ref{thm:equivalence}. We consider the following condition on $(\Delta t, h)$:
\begin{equation}\label{cond:CFL3}\tag{CFL}
\Delta t \leq \frac{h^2}{2d (1-\theta)\sigma} ,\qquad h \leq \frac{2(1-\theta)\sigma}{M}.
\end{equation}

\begin{rem}\label{rem:fp}
  Let us reformulate the explicit part of \eqref{eq:theta_mfg2}(iii) by isolating $m(t+1/2,x)$:
      \begin{equation}\label{eq:fp_exp}
      \begin{split}
          m(t+1/2 ,x) = \Big(1-\frac{2d(1-\theta )\sigma\Delta t}{h^2}\Big) m(t,x)  + \Delta t\sum_{i=1}^d\Big(\frac{(1-\theta )\sigma}{h^2} - \frac{v_i(t,x + he_i)}{2h} \Big) m(t,x+he_i) \\
          + \Delta t\sum_{i=1}^d\Big(\frac{(1-\theta )\sigma}{h^2} + \frac{v_i(t,x - he_i)}{2h} \Big) m(t,x-he_i) .
      \end{split}
      \end{equation}
The coefficients preceding $m(t,x)$ and $m(t,x\pm he_i)$ in \eqref{eq:fp_exp} are affine functions with respect to $v(t,x)$ and $v(t,x\pm he_i)$ respectively, and these coefficients are positive under the condition \eqref{cond:CFL3} since $M$ is an upper bound of $\|v\|_{\infty,\infty}$.
Moreover, summing \eqref{eq:fp_exp} over $x$ yields that $\sum_{x\in S} m(t+1/2,x) = \sum_{x\in S}m(t,x)$. Therefore, under the condition \eqref{cond:CFL3}, if $m(t)\in \mathcal{P}(S)$, then $m(t+1/2) \in \mathcal{P}(S)$. Since $m(t+1)$ is the solution of an implicit scheme for the heat equation (with source term $m(t+1/2,x)$), we have that $m(t+1)\in \mathcal{P}(S)$ if $m(t+1/2)\in \mathcal{P}(S)$, by Lemma \ref{lm:implicit}. In other words, probability distributions on $S$ are preserved by the discrete Fokker-Planck equation under the condition \eqref{cond:CFL3}.
\end{rem}

\begin{rem}\label{rem:theta}
 Let us discuss the choice of $\theta$ in the theta-scheme \eqref{eq:theta_mfg2}. If we set $\theta =1$, we cannot guarantee the positivity of the coefficients preceding $m(t,x\pm he_i)$ in \eqref{eq:fp_exp}. As a result, we cannot use the same argument presented in Remark \ref{rem:fp} to ensure the preservation of probability distributions of the discrete FP equation. On the other hand, to obtain an energy estimate ($\ell^2$-stability) of the discrete FP equation, we require that $\theta > 1/2$, as demonstrated in Proposition \ref{prop:energy}.
\end{rem}

\begin{thm}\label{thm:main}
Let Assumptions \ref{ass:continuous} and \ref{ass:sol+} hold true. Let $\theta \in (1/2,1)$ and let $(\Delta t, h)$ satisfy the condition \eqref{cond:CFL3}. Then, the theta-scheme \eqref{eq:theta_mfg2} has a unique solution $(u_h, v_h, m_h)$. Moreover, there exists a constant $C>0$, independent of $\Delta t$ and $h$, such that
\begin{equation*}
\|u_h - u^*_h \|_{\infty,\infty}  + \|m_h - m^{*}_h \|_{\infty,1} \leq C h^r,
\end{equation*}
{ where  $u^{*}_h, m^{*}_h \in \R(\tilde{\mathcal{T}} \times S)$ are defined by $u^{*}_h(t,x)= u^{*}(t \Delta t, x)$ and $m^{*}_h(t)= \mathcal{I}_h (m^{*}(t\Delta t))$.}
\end{thm}

The proof of Theorem \ref{thm:main} is given in Section \ref{sec:proof}.

\section{General properties of discrete mean field games}

We consider in this section a general class of discrete time and finite state space mean field games, for which we establish the existence and uniqueness of a solution as well as a fundamental inequality. We will show in Section \ref{sec:stability} that the theta-scheme falls into this class of problems.

\subsection{Notations and assumptions}

In this section, the state space $S$ is an arbitrary discrete set in $\mathbb{R}^d$, not necessarily a discretization of $\mathbb{T}^d$.
Let us introduce the set of discrete curves of probability measures and the set of transition processes, defined by
\begin{align*}
     \mathcal{P}(\tilde{\mathcal{T}},S)  &=  \left\{  m \in \mathbb{R}( \tilde{\mathcal{T}}\times S) \, \Big|     \, \forall t \in \tilde{\mathcal{T}} , \, m(t,\cdot) \in \mathcal{P}(S)  \right \},
     \\
     \Pi(\mathcal{T}, S) & = \left\{ \pi \in \R(\mathcal{T} \times S \times S) \, \Big| \, \forall (t,x) \in \mathcal{T}\times S, \, \pi(t,x,\cdot) \in \mathcal{P}(S) \right\}.
\end{align*}


\begin{rem} \label{rem:charac_process}
Any $\pi \in \R(\mathcal{T} \times S \times S)$ is a transition process if and only if for any $m \in \mathcal{P}(S)$ and for any $t \in \mathcal{T}$, we have $m' \in \mathcal{P}(S)$, for
$m'(y)= \sum_{x \in S} \pi(t,x,y) m(x)$, for all $y \in S$.
\end{rem}

We introduce now a running cost $\ell$, a coupling cost $f$, an initial condition $m_0$ and a terminal cost $g$, where
\begin{align*}
\ell \colon \mathcal{T}\times S \times \mathbb{R}^d \rightarrow \mathbb{R} ,
\qquad
f \colon \mathcal{T}\times S \times \mathbb{R}(S) \rightarrow \mathbb{R},
\qquad
m_0 \in \mathcal{P}(S),
\qquad
g \in \mathbb{R}(S).
\end{align*}
In this section, $\ell$, $f$, $m_0$, and $g$ are considered independently of the definition \eqref{eq:grid1}. We will consider again definition \eqref{eq:grid1} in the next section when we interpret the theta-scheme as a discrete MFG.

To formulate the discrete MFG system, we need a control bound $\bar{D}>0$. The admissible control space, denoted by $\R^d_{\bar{D}}(\mathcal{T}\times S)$, is the set of all elements $v\in \R^d(\mathcal{T}\times S)$ such that $\|v\|_{\infty,\infty}\leq \bar{D}$. The probability of the motion from one state $x\in S$ to another state $y\in S$ at a time $t\in \mathcal{T}$ under some control $v\in \R^d_{\bar{D}}(\mathcal{T}\times S)$ is given by
\begin{equation*}
  \pi[v](t,x,y) \coloneqq  \pi(t,x,y, v(t,x)),
\end{equation*}
where $\pi$ is a function from $\mathcal{T}\times S\times S\times \R^d$ to $\R$. We assume that $\pi[v]$ is a transition process for any admissible control $v$, i.e.,
\begin{equation}\label{eq:ass}
    \pi[v] \in \Pi(\mathcal{T}, S), \qquad \forall v\in \R^d_{\bar{D}}(\mathcal{T}\times S).
\end{equation}

For any $D \in (0,\infty]$, we denote by $\ell^D \colon \mathcal{T} \times S \times \R^d \rightarrow \R \cup \{ \infty \}$ the function defined by
\begin{equation} \label{eq:lD}
\ell^D(t,x,v)
=
\begin{cases}
\begin{array}{cl}
\ell(t,x,v), & \text{ if } \| v \| \leq D, \\
\infty, & \text{ otherwise.}
\end{array}
\end{cases}
\end{equation}
When $D= \infty$, $\ell^D= \ell$.
The Hamiltonian $H^D$ is defined as follows:
\begin{equation}\label{eq:HD}
    H^D(t,x,p) = \sup_{v\in \mathbb{R}^d} \ \langle -p, v\rangle - \ell^D(t,x,v)
    =
    \sup_{ v\in \mathbb{R}^d, \, \| v \| \leq D} \ \langle -p, v\rangle - \ell(t,x,v).
\end{equation}
We consider the following assumptions on the previous data.

\begin{ass2} \label{ass1}
\begin{enumerate}
\item \emph{Regularity.} There exist positive constants $L_{\ell}$, $L_g$, $L_f$, and $L_{f'}$ such that for any $t \in \mathcal{T}$, for any $v \in R^d$, and for any $m \in \mathcal{P}(S)$, the functions
$\ell(t,\cdot,v)$, $g(\cdot)$, and $f(t,\cdot,m)$ are resp.\@ $L_{\ell}$, $L_g$, and $L_f$-Lipschitz continuous, i.e.
\begin{equation*}
\begin{split}
  |\ell(t,x_1,v) - \ell(t,x_2,v)| \leq {} & L_{\ell} \|x_1 - x_2\|, \\
    |g(x_1) - g(x_2)| \leq {} & L_g\|x_1-x_2\|, \\
    |f(t,x_1,m) - f(t,x_2,m)| \leq {} & L_f  \|x_1 - x_2\|,
    \end{split}
\end{equation*}
for all $x_1$ and $x_2$ in $S$. Moreover, the function $f(t,x,\cdot)$ is $L_f'$-Lipschitz w.r.t.\@ $m$ for the $\|\cdot\|_2$ norm , i.e., for all $m_1$ and $m_2$ in $\mathcal{P}(S)$,
    \begin{equation*}
       |f(t,x,m_1) - f(t,x,m_2)| \leq  L_f' \|m_1 - m_2\|_2.
    \end{equation*}
\item \emph{Strong convexity.} There exist $\alpha > 0$ such that for any $t \in \mathcal{T}$ and for any $x \in S$, the function $\ell(t,x,\cdot)$ is $\alpha$-strongly convex, i.e.,
   \begin{equation*}
        \ell(t,x,v_2) \geq \ell(t,x,v_1) + \langle p, v_2-v_1\rangle + \frac{\alpha}{2}\|v_2-v_1\|^2,
    \end{equation*}
    for all $v_1$ and $v_2$ in $\R^n$ and for all $p \in  \partial_p \ell(t,x,v_1)$.
\item \emph{Monotonicity.} For any $t \in \mathcal{T}$, for any $m_1$ and $m_2$ in $\mathcal{P}(S)$,
    \begin{equation*}
        \sum_{x\in S} \Big( f(t,x,m_1)-f(t,x,m_2)\Big) (m_1(x)-m_2(x) ) \geq 0.
    \end{equation*}
\end{enumerate}
\end{ass2}

\begin{lem} \label{lem:hD_diff}
Let $D \in (0,\infty]$. The following holds true.
\begin{enumerate}
\item The Hamiltonian $H^D$ is continuously differentiable with respect to $p$.
\item For any $t \in \mathcal{T}$, for any $x \in S$, and for any $v \in \R^d$, we have $H^D(t,x,p)= -\langle p,v\rangle - \ell^D(t,x,v)$ if and only if $v= - H_p^D(t,x,p)$.
\item The partial derivative $H_p^D$ is $\frac{1}{\alpha}$-Lipschitz continuous with respect to $p$.
\item For any $t \in \mathcal{T}$, for any $x \in S$, for any $v \in \R^d$, and for any $p_0 \in \partial_v \ell(t,x,0)$,
\begin{equation} \label{eq:bound:Hp}
\big\| H_p^D(t,x,p) \big\| \leq \frac{1}{\alpha} \Big( 2 \|p_0 \| + \|  p \| \Big).
\end{equation}
\end{enumerate}
\end{lem}

The proof is given in the Appendix \ref{Appendix:A}. 
A direct consequence of Lemma \ref{lem:hD_diff} is the following.

\begin{cor} \label{coro:equi_H}
Let $(t,x,p) \in  \mathcal{T} \times S \times \R^d$. Let $p_0 \in \partial_v \ell (t,x,0)$. Let $D_1$ and $D_2 \in (0,\infty]$ be such that $D_i \geq \frac{1}{\alpha} \big( 2 \| p_0 \| + \| p \| \big)$, for $i=1,2$. Then
\begin{equation*}
H^{D_1}(t,x,p)= H^{D_2}(t,x,p)
\quad \text{and} \quad
H_p^{D_1}(t,x,p)= H_p^{D_2}(t,x,p).
\end{equation*}
\end{cor}

\begin{lem} \label{lem:hD_quad}
Let $D \in (0,\infty]$, let $t \in \mathcal{T}$ and let $x \in S$. For any $v$, for any $\bar{p} \in \R^d$, for any $m \geq 0$ and for any $\bar{m} \in \R$, it holds that
\begin{equation}\label{eq:d_lm}
    \ell^D(t,x,v)m - \ell^D(t,x,\bar{v}) \bar{m} \geq - H^D(t,x,\bar{p})(m-\bar{m}) - \langle \bar{p}, m v - \bar{m}\bar{v} \rangle + \frac{\alpha}{2} \| v -\bar{v}\|^2 m,
\end{equation}
where $\bar{v}= -H_p^D(t,x,\bar{p})$.
\end{lem}

The proof is given in the Appendix \ref{Appendix:A}. 

\subsection{The discrete MFG model}

The discrete MFG model of interest in this section is a coupled system of three variables: a value function $u \in \R(\tilde{\mathcal{T}} \times S)$, a policy $v \in \R_{\bar{D}}^d(\mathcal{T} \times S)$, and a curve of probability distributions $m \in \R(\tilde{\mathcal{T}} \times S)$. It consists of a Kolmogorov equation, a dynamic programming equation, and a feedback relation.

\begin{itemize}
\item Given $v \in \R_{\bar{D}}^d(\mathcal{T} \times S)$, denote by $\textbf{FP}(v) \in \R(\mathcal{T} \times S)$ the solution $m$ to the Kolmogorov equation
\begin{equation}\label{eq:kolmogorov}
\begin{cases}
\begin{array}{rll}
m(t+1,y) = & \!\!\! \sum_{x\in S} \pi[v](t,x,y)m(t,x),  \quad & \forall (t,y) \in \mathcal{T}\times S, \\
m(0,x)  = & \!\!\! m_0(x), & \forall x \in S.
\end{array}
\end{cases}
\end{equation}
\item Given $\mu \in \mathcal{P}(\tilde{\mathcal{T}},S)$, denote by $\textbf{HJB}(\mu) \in \R(\tilde{\mathcal{T}} \times S)$ the solution $u$ to the dynamic programming equation
\begin{equation}\label{eq:DP}
\begin{cases}
\begin{array}{rll}
      u(t,x) = & \! \! \!{\displaystyle \inf_{\omega\in \mathbb{R}^d}} \Big( \tilde{\ell}_{\mu}^{\bar{D}}(t,x,\omega) \Delta t + \sum_{y\in S} \pi(t,x,y,\omega )u(t+1, y) \Big) ,\quad &\forall(t,x)\in \mathcal{T}\times S; \\
      u(T,x) = & \! \! \! g(x) , &\forall x\in S,
\end{array}
\end{cases} 
\end{equation}
where $\tilde{\ell}_{\mu}^{\bar{D}}(t,x,\omega) = \ell^{\bar{D}}(t, x, \omega) + f(t, x,\mu(t))$. 
\item Given $u \in \R(\tilde{\mathcal{T}} \times S)$, denote by $\textbf{V}(u)$ the policy $v$ defined by
\begin{equation}\label{eq:OS}
     v(t,x)  = \argmin_{\omega\in \mathbb{R}^d}
     \Big( \ell^{\bar{D}}(t,x,\omega)\Delta t + \sum_{y\in S} \pi(t,x,y,\omega )u(t+1, y) \Big), \quad \forall(t,x)\in \mathcal{T}\times S.
\end{equation}
The uniqueness of the minimizer in the above definition is a consequence of Lemma \ref{lem:hD_diff}.
\end{itemize}

The discrete MFG consists in finding a triplet $(u,v,m)$ such that $u= \textbf{HJB}(m)$, $v= \textbf{V}(u)$, and $m = \textbf{FP}(v)$. This is equivalent to find a fixpoint to the map $\phi$, defined by
\begin{equation*} \label{def:phi}
\phi \colon m \in \mathcal{P}(\tilde{\mathcal{T}}, S) \mapsto \textbf{FP} \circ \textbf{V} \circ \textbf{HJB} (m) \in \mathcal{P}(\tilde{\mathcal{T}},S).
\end{equation*}
It is easy to verify that $\phi$ is indeed valued in $\mathcal{P}(\tilde{\mathcal{T}} \times S)$. Let $m \in \mathcal{P}(\tilde{\mathcal{T}} \times S)$ and let $v=\textbf{V} \circ \textbf{HJB} (m) $. By definition, $\| v \|_{\infty,\infty} \leq {\bar{D}}$. Therefore, by assumption \eqref{eq:ass}, $\pi[v]$ is a transition process. Then $\textbf{FP}(v) \in \mathcal{P}(\tilde{\mathcal{T}} , S)$, by Remark \ref{rem:charac_process}.

The discrete MFG can be formulated as the following coupled system: for all $(t,x) \in \mathcal{T} \times S$,
\begin{equation}\label{eq:mfgt}
\left\{\begin{array}{cl}
\mathrm{(i)} \ & u(t,x) = \inf_{\omega\in \mathbb{R}^d}\tilde{\ell}_{m}^{\bar{D}}(t,x,\omega) \Delta t + \sum_{y\in S} \pi(t,x,y,\omega )u(t+1, y); \\[0.7em]
\mathrm{(ii)} \ & v(t,x) = \argmin_{\omega\in \mathbb{R}^d}\ell^{\bar{D}}(t,x,\omega)\Delta t + \sum_{y\in S} \pi(t,x,y,\omega )u(t+1, y);
\\[0.7em]
\mathrm{(iii)} \ & m(t+1,x) = \sum_{y\in S} \pi[v](t,y,x)m(t,y);\\[0.7em]
\ \mathrm{(iv)} \ \ & m(0,x)=m_{0}(x), \quad u(T,x)=g(x).
\end{array}\right.
\end{equation}

As mentioned in Remark \ref{rem:fp}, the coefficients preceding $m(t,x\pm he_i)$ in \eqref{eq:fp_exp} are affine functions with respect to $v(t,x\pm he_i)$. Furthermore, $m(t+1)$ can be seen as a linear function of $m(t+1/2)$ independent of $v$ from the implicit part of \eqref{eq:theta_mfg2}(iii). Therefore, in the theta-scheme \eqref{eq:theta_mfg2}, we can express $m(t+1,x)$ as a linear combination of $m(t,y)$ for $y\in S$, where the coefficients preceding $m(t,y)$ are affine functions with respect to $v(t,y)$.
Comparing this with the coefficients $\pi[v](t,y,x) = \pi(t,y,x,v(t,y))$ in \eqref{eq:mfgt}(iii), in order to study \eqref{eq:theta_mfg2} as a particular case of \eqref{eq:mfgt}, we find it convenient to consider $\pi(t,x,y,\omega)$ in an affine form of $\omega$, i.e., 
\begin{equation}\label{eq:pi}
    \pi(t,x,y,\omega) = \pi_0(t,x,y) + \Delta t \langle \pi_1(t,x,y), \omega \rangle, \qquad \forall (t,x,y,\omega) \in \mathcal{T}\times S^2\times \R^d,
\end{equation}
where $\pi_0 \in \mathbb{R}(\mathcal{T}\times S\times S)$ and $\pi_1 \in \mathbb{R}^d(\mathcal{T}\times S\times S)$. 
The exact formulas for $\pi_0$ and $\pi_1$ associated with \eqref{eq:theta_mfg2} are given in \eqref{eq:Pi0}-\eqref{eq:Pi1}.

In the sequel of this section, we consider $\pi$ given by \eqref{eq:pi}. We make the following assumption on $\pi_0$ and $\pi_1$.

\begin{ass2} \label{ass3}
The elements $\pi_0$ and $\pi_1$ satisfy the following condition:
 \begin{equation*} 
\begin{cases}
\begin{array}{ll}
      \pi_0(t,x,\cdot) \in \mathcal{P}(S), \quad  &\forall (t,x) \in \mathcal{T}\times S, \\[0.4em]
 \sum_{y\in S}\pi_1 (t,x,y) = 0 , \quad  &\forall (t,x) \in \mathcal{T}\times S, \\[0.4em]
      \pi_0(t,x,y) \geq  \Delta t \bar{D} \| \pi_1(t,x,y)\| , \qquad &\forall (t,x,y) \in \mathcal{T}\times S\times S.
\end{array}
\end{cases}
\end{equation*}
\end{ass2}
\begin{lem}\label{lem:trivial}
For $\pi$ given by \eqref{eq:pi}, Assumption \ref{ass3} is equivalent to \eqref{eq:ass}.
\end{lem}

The proof of the previous lemma is left to the reader.

Thanks to \eqref{eq:pi}, we can simplify \eqref{eq:mfgt} (i)-(ii) with the help of $H^{\bar{D}}$ (defined by \eqref{eq:HD}). Let us define $p_0$, $p_1$, $q_0$, and $q_1$ as follows: for all $(t,x) \in \mathcal{T}\times S$,
\begin{align}
     p_0(t,x) = {} &\sum_{s\in S} \pi_0(t,x,s) u(t+1,s), & p_1(t,x) = {} & \sum_{s\in S} \pi_1 (t,x,s) u(t+1,s);  \label{p}\\
     q_0(t,x) = {} & \sum_{s\in S} \pi_0(t,s,x) m(t,s),  & q_1[v](t,x) = {} & \sum_{s\in S} \langle\pi_1 (t,s,x), v(t,s)  m(t,s)\rangle. \label{q}
\end{align}
Observe that the dependence of $p_0$ and $p_1$ with respect to $u$ is not explicitly mentioned, similarly, the dependence of $q_0$ and $q_1$ with respect to $m$ and $v$ is not explicitly mentioned and will be clear from the context.
Then system \eqref{eq:mfgt} equivalently writes: for all $(t,x) \in \mathcal{T} \times S$,
\begin{equation}\label{eq:dmfg}\tag{DMFG}
\left\{\begin{array}{cll}
\mathrm{(i)} \ & u(t,x) =  \big(-H^{\bar{D}}[p_1](t,x)+f(t,x,m(t))\big) \Delta t + p_0(t,x) ;\\[0.8em]
\mathrm{(ii)} \  & v(t,x) = -H_p^{\bar{D}}[p_1](t,x);\\[0.8em]
\mathrm{(iii)} \ & m(t+1,x) = q_0(t,x) + \Delta t q_1[v](t,x);\\[0.8em]
\ \mathrm{(iv)} \ \ & m(0,x)=m_{0}(x), \quad u(T,x)=g(x).
\end{array}\right.
\end{equation}

\begin{thm}[Existence]\label{thm1}
 Under Assumptions \ref{ass1} and \ref{ass3}, \eqref{eq:dmfg} has at least one solution. Furthermore, if $(\bar{u},\bar{v},\bar{m})$ is a solution of \eqref{eq:dmfg}, then $\bar{m}\in\mathcal{P}(\tilde{\mathcal{T}},S)$.
\end{thm}

\begin{proof}[Proof (first part)]
We equip the finite-dimensional space $\mathbb{R}(\tilde{\mathcal{T}}\times S)$ with the norm $\|\cdot\|_{\infty,1}$.
The set $\mathcal{P}(\tilde{\mathcal{T}},S)$ is non-empty, convex, and compact. In order to prove the existence of a solution, we need to show the existence of fixpoint for the map $\phi$, defined in \eqref{def:phi}. By the Brouwer fixed-point theorem, it suffices to show that $\phi$ is a continuous mapping,
which we do in the appendix (page \pageref{proof:thm1}).
\end{proof}

\subsection{A fundamental inequality}

Let us define a perturbed version of \eqref{eq:dmfg} with additional terms $(\eta,\delta) \in \mathbb{R}^2(\mathcal{T}\times S)$ in the right-hand side: for all $(t,x) \in \mathcal{T} \times S$,
\begin{equation}\label{eq:dmfgd}\tag{PDMFG}
\left\{\begin{array}{cll}
\mathrm{(i)} \ & u(t,x) =  \Big(-H^{\bar{D}}[p_1](t,x)+f(t,x,m(t))\Big) \Delta t + p_0(t,x) + \eta(t,x);\\[0.8em]
\mathrm{(ii)} \  & v(t,x) = -H_p^{\bar{D}}[p_1](t,x);\\[0.8em]
\mathrm{(iii)} \ & m(t+1,x) = q_0(t,x) + \Delta t q_1[v](t,x) + \delta(t,x);\\[0.8em]
\ \mathrm{(iv)} \ \ & m(0,x)=m_{0}(x), \quad u(T,x)=g(x).
\end{array}\right.
\end{equation}
The fundamental inequality proved in the next proposition is an essential tool in the stability analysis for the system \eqref{eq:dmfg}.

\begin{prop}[Fundamental inequality]\label{prop:error_fund}
Let Assumptions \ref{ass1} and \ref{ass3} hold true. Let $(\bar{u},\bar{v},\bar{m})$ be a solution of \eqref{eq:dmfg} and let $(u,v,m)$ satisfy \eqref{eq:dmfgd} with $m\geq 0$. Then, the following inequality holds:
\begin{equation}\label{eq:fund}
      \frac{\Delta t \alpha}{2}
      \sum_{t \in\mathcal{T}}\sum_{x\in S} \| (v - \bar{v})(t,x) \|^2(m+\bar{m})(t,x)
      \leq  \sum_{t\in\mathcal{T}} \sum_{x\in S} (u-\bar{u})(t+1,x) \delta(t,x) + (\bar{m}-m) (t,x)\eta (t,x).
\end{equation}
\end{prop}

This fundamental inequality is of the same nature as the one established in \cite[Sec.\@ 3.3]{achdou2013mean}. We provide two different proofs of Proposition \ref{prop:error_fund} in the next subsection.
The fundamental inequality allows us to show the uniqueness of the solution to \eqref{eq:dmfg}.

\begin{lem}[Uniqueness]\label{lemma:unique}
  Under Assumptions \ref{ass1}-\ref{ass3}, \eqref{eq:dmfg} has a unique solution.
\end{lem}

\begin{proof}
The existence result was already established in Theorem \ref{thm1}. Let $(u_1,v_1,m_1)$ and $(u_2,v_2,m_2)$ be two solutions of \eqref{eq:dmfg}. By Theorem \ref{thm1}, $m_1\geq$ and $m_2\geq 0$. 
Viewing $(u_2,v_2,m_2)$ as a solution to \eqref{eq:dmfgd} with $(\eta,\delta)=(0,0)$, we deduce from the fundamental inequality that
\begin{equation*}
\| v_1(t,x)-v_2(t,x) \| (m_1(t,x)+m_2(t,x))= 0.
\end{equation*}
Thus for any $(t,x)\in \mathcal{T}\times S$, either $v_1(t,x) = v_2(t,x)$, or $m_1(t,x)=m_2(t,x)=0$. Let $\mu=m_1-m_2$, then $\mu$ satisfies the following equation: for any $(t,x) \in \mathcal{T}\times S$,
\begin{equation*}
 \begin{cases}
 \begin{array}{rl}
       \mu(t+1,x) = & \! \! \! \sum_{s\in S} \pi[v_1](t,s,x) \mu(t,s) +  \Delta t \sum_{s\in S} \langle \pi_1(t,s,x), (v_1-v_2) m_2(t,s) \rangle, \quad \\
       \mu(0,x) = & \! \! \! 0.
       \end{array}
 \end{cases}
 \end{equation*}
It follows by induction that $\mu =0$, i.e.\@ $m_1=m_2$. 
Then $u_1= \textbf{HJB}(m_1)= \textbf{HJB}(m_2)= u_2$ and $v_1= \textbf{V}(u_1)= \textbf{V}(u_2)= v_2$, which concludes the proof.
\end{proof}

\subsection{Two proofs of the fundamental inequality}
{In this subsection, $(u,v,m)$ is a solution of \eqref{eq:dmfgd} and $(p_0,p_1)$ is defined by \eqref{p}.}
Let $(\bar{u},\bar{v},\bar{m})$ be a solution to \eqref{eq:dmfg}.
Let $(\bar{p}_0, \bar{p}_1)$ be defined by \eqref{p}, for the triplet $(\bar{u},\bar{v},\bar{m})$.
The following sum-by-parts formulas will be used in both two methods of proof. For all $t \in \mathcal{T}$,
\begin{align}
    \sum_{x\in S} \bar{p}_0 \bar{m} (t,x) + \Delta t\langle \bar{p}_1,  \bar{m}\bar{v} \rangle(t,x) & = \sum_{y\in S} \bar{u}(t+1,y) \bar{m} (t+1,y) ; \label{eq:fund1}\\
    \sum_{x\in S} \bar{p}_0 m (t,x) + \Delta t\langle \bar{p}_1, m v \rangle(t,x) & = \sum_{y\in S} \bar{u}(t+1,y) m (t+1,y) -  \sum_{y\in S} \bar{u}(t+1,y) \delta(t,y); \label{eq:fund2}\\
    \sum_{x\in S} {p}_0 \bar{m} (t,x) + \Delta t\langle {p}_1, \bar{m}\bar{v} \rangle(t,x) & = \sum_{y\in S} {u}(t+1,y) \bar{m} (t+1,y);\label{eq:fund3}\\
    \sum_{x\in S} {p}_0 m (t,x) + \Delta t\langle {p}_1, m v\rangle(t,x) & = \sum_{y\in S} {u}(t+1,y) m (t+1,y) -  \sum_{y\in S} {u}(t+1,y) \delta(t,y). \label{eq:fund4}
\end{align}
For proving \eqref{eq:fund1}, one simply needs to multiply the first equation in \eqref{p} by $\bar{m}(t,x)$, to multiply the second equation in \eqref{p} by $\bar{m}\bar{v}(t,x)$ and to sum the results over $x$. This yields
\begin{equation*}
\begin{split}
    \sum_{x\in S} \bar{p}_0 \bar{m} (t,x) + \Delta t\langle \bar{p}_1,  \bar{m}\bar{v} \rangle(t,x) = & \sum_{x\in S} \sum_{s\in S}  \bar{u}(t+1,s) \Big(\pi_0(t,x,s)  \bar{m}(t,x)+ \Delta t \langle \pi_1 (t,x,s), \bar{v}\bar{m}(t,s)\rangle \Big).
\end{split}
\end{equation*}
Then \eqref{eq:fund1} follows from \eqref{eq:dmfg}-(iii).
The proofs of the other three equations can be obtained similarly. We provide now two different proofs of Proposition \ref{prop:error_fund}.

\paragraph{1. Direct method} We follow \cite{achdou2013mean}. Summing the difference of \eqref{eq:fund2} and \eqref{eq:fund1} over $t\in \mathcal{T}$, we get
\begin{equation}\label{eq:fund5}
    \sum_{t=1}^{T}\sum_{x\in S} \bar{u}(m-\bar{m})(t,x) = \sum_{t\in\mathcal{T}}\sum_{x\in S} \bar{p}_0(m-\bar{m})(t,x) + \Delta t\langle \bar{p}_1, m v -\bar{m}\bar
    v\rangle(t,x) + \bar{u}(t+1,x)\delta(t,x).
\end{equation}
In addition, summing the difference of \eqref{eq:fund4} and \eqref{eq:fund3} over $t\in \mathcal{T}$, we get
\begin{equation}\label{eq:fund6}
    \sum_{t=1}^{T}\sum_{x\in S} {u}(m-\bar{m})(t,x) = \sum_{t\in\mathcal{T}}\sum_{x\in S} {p}_0(m-\bar{m})(t,x) + \Delta t\langle {p}_1, m v -\bar{m}\bar
    v\rangle(t,x) + {u}(t+1,x)\delta(t,x).
\end{equation}
Taking the difference of \eqref{eq:fund6} and \eqref{eq:fund5}, we have
\begin{equation}\label{eq:fund7}
\begin{split}
   &  \sum_{t=1}^{T}\sum_{x\in S} (u-\bar{u})(m-\bar{m})(t,x) \\
   & \quad = \sum_{t\in\mathcal{T}}\sum_{x\in S} (p_0-\bar{p}_0)(m-\bar{m})(t,x) + \Delta t\langle {p}_1 - \bar{p}_1, m v -\bar{m}\bar
    v\rangle(t,x) + (u-\bar{u})(t+1,x)\delta(t,x).
\end{split}
\end{equation}
Moreover, taking the difference of  \eqref{eq:dmfgd} (i) and \eqref{eq:dmfg} (i), multiplying the result by $m-\bar{m}$, summing over $(t,x)\in\mathcal{T}\times S$, we obtain that
\begin{equation}\label{eq:fund9}
\begin{split}
     & \sum_{t\in\mathcal{T}}\sum_{x\in S} (u-\bar{u})(m-\bar{m})(t,x) \\
    & \quad =  \sum_{t\in\mathcal{T}}\sum_{x\in S} (p_0-\bar{p}_0)(m-\bar{m})(t,x) + \Delta t\big(H^{\bar{D}}[\bar{p}_1] - H^{\bar{D}}[p_1]\big)(m-\bar{m})(t,x) \\
      & \quad \qquad + \sum_{t\in\mathcal{T}}\sum_{x\in S}\Delta t\big(f(t,x,m(t)) - f(t,x,\bar{m}(t))\big)(m-\bar{m})(t,x) + \eta(m-\bar{m})(t,x).
\end{split}
\end{equation}
Comparing  \eqref{eq:fund7} and \eqref{eq:fund9} and using the relations $v=-H_p^{\bar{D}}[p_1]$, $\bar{v} = -H_p^{\bar{D}}[\bar{p}_1]$, we obtain the following equality:
\begin{equation}\label{eq:fund10}
\begin{split}
     &\Delta t \sum_{t\in\mathcal{T}}\sum_{x\in S}  m\Big( H^{\bar{D}}[\bar{p}_1] - H^{\bar{D}}[p_1] - \langle H_p^{\bar{D}}[p_1] , \bar{p}_1 - p_1\rangle \Big)(t,x) \\
     & \qquad +  \Delta t \sum_{t\in\mathcal{T}}\sum_{x\in S}  \bar{m}\Big(H^{\bar{D}}[p_1] - H^{\bar{D}}[\bar{p}_1] - \langle H_p^{\bar{D}}[\bar{p}_1] , p_1-\bar{p}_1\rangle \Big)(t,x)\\
    & \qquad  +  \Delta t\sum_{t\in\mathcal{T}}\sum_{x\in S}\big(f(t,x,m(t)) - f(t,x,\bar{m}(t))\big)(m-\bar{m})(t,x)\\
    & \qquad \qquad =   \sum_{t\in\mathcal{T}}\sum_{x\in S} (u-\bar{u})(t+1,x)\delta(t,x) +  (\bar{m}-m)(t,x)\eta(t,x).
\end{split}
\end{equation}
Since $H^{\bar{D}}$ is convex and $H_p^{\bar{D}}$ is $1/\alpha$-Lipschitz, we obtain with inequality \eqref{eq:nesterov} (see \cite[Thm.\@ 2.1.5]{nesterov2018lectures}) that
\begin{align*}
     H^{\bar{D}}[\bar{p}_1] - H^{\bar{D}}[p_1] - \langle H^{\bar{D}}_p[p_1] , \bar{p}_1 - p_1\rangle \geq \frac{\alpha}{2} \big\| H^{\bar{D}}_p[p_1] - H^{\bar{D}}_p[\bar{p}_1] \big\|^2 =\frac{\alpha}{2}\|v-\bar{v}\|^2;\\
     H^{\bar{D}}[p_1] - H^{\bar{D}}[\bar{p}_1] - \langle H^{\bar{D}}_p[\bar{p}_1] , p_1 - \bar{p}_1\rangle \geq \frac{\alpha}{2} \big\| H^{\bar{D}}_p[p_1] - H^{\bar{D}}_p[\bar{p}_1] \big\|^2 =\frac{\alpha}{2}\|v-\bar{v}\|^2.
\end{align*}
We substitute the last two inequalities into \eqref{eq:fund10}. Then, inequality \eqref{eq:fund} follows from the non-negativity of $m$ and $\bar{m}$ and the monotonicity of $f$ in Assumption \ref{ass1}.

\paragraph{2. Variational method} Let us define a ``relative" potential function $\tilde{J}_{\bar{m}}  \colon  \mathbb{R}^d(\mathcal{T}\times S)\times \mathcal{P}(\mathcal{\tilde{T}}, S) \rightarrow \mathbb{R}$, 
\begin{equation*}
     \tilde{J}_{\bar{m}}(v,m) = \Delta t \sum_{t\in \mathcal{T}}\sum_{x\in S} m(t,x) \Big( \ell^{\bar{D}}(t,x, v(t,x)) + f(t,x,\bar{m}(t)) \Big) +  \sum_{x\in S} g(x) m(T,x).
\end{equation*}
Note that in the above function $\tilde{J}_{\bar{m}}$,
the third variable of $f$ is fixed to $\bar{m}$.
The second proof of Proposition \ref{prop:error_fund} consists in proving a lower bound and an upper bound of $\tilde{J}_{\bar{m}}(v,m) - \tilde{J}_{\bar{m}}(\bar{v},\bar{m})$, from which the fundamental inequality directly follows.

\smallskip

\noindent \textbf{Step 1.} Let us prove that 
\begin{equation*}
    \tilde{J}_{\bar{m}}(v,m) - \tilde{J}_{\bar{m}}(\bar{v},\bar{m}) \geq  \sum_{t\in\mathcal{T}} \sum_{x\in S}\bar{u}(t+1,x) \delta(t,x) +  \Delta t  \sum_{t\in\mathcal{T}} \sum_{x\in S}\frac{\alpha}{2}\|v - \bar{v}\|^2 m (t,x).
\end{equation*}
By Lemma \ref{lem:hD_quad}, we have
\begin{equation*}
\begin{split}
    \left(\ell^{\bar{D}}[v] m - \ell^{\bar{D}}[\bar{v}]\bar{m}\right) \Delta t \geq {} & \left(- H^{\bar{D}}[\bar{p}_1] (m-\bar{m}) - \langle \bar{p}_1, mv - \bar{m}\bar{v} \rangle + \frac{\alpha}{2} \|v - \bar{v}\|^2 m\right) \Delta t\\ 
     ={} & \left( \bar{u} - \bar{p}_0 - \Delta t f(t,x,\bar{m}(t)) \right)(m-\bar{m}) - \Delta t\langle  \bar{p}_1, mv - \bar{m}\bar{v} \rangle + \Delta t\frac{\alpha}{2} \|v - \bar{v}\|^2 m\\
     ={} &\bar{u} (m-\bar{m}) + \left( \bar{p}_0 \bar{m} +\Delta t \langle \bar{p}_1 , \bar{m} \bar{v} \rangle\right) -\left( \bar{p}_0 {m} + \Delta t\langle \bar{p}_1 , {m} {v} \rangle\right)\\
     &\quad - \Delta tf(t,x,\bar{m}(t)) (m-\bar{m}) + \Delta t \frac{\alpha}{2} \|v - \bar{v}\|^2 m.
\end{split}
\end{equation*}
It follows that
\begin{equation*}
\begin{split}
    &\tilde{J}_{\bar{m}}(v,m) - \tilde{J}_{\bar{m}}(\bar{v},\bar{m}) \\
   & \qquad = \Delta t  \sum_{t\in\mathcal{T}}\sum_{x\in S}\left(\ell^{\bar{D}}[v] m - \ell^{\bar{D}}[\bar{v}]\bar{m} + f(t,x,\bar{m}(t))(m-\bar{m})\right)(t,x)  +  \sum_{x\in S} g(x) (m-\bar{m})(T,x) \\
  & \qquad  \geq  \sum_{t\in\mathcal{T}}\sum_{x\in S}\bar{u} (m-\bar{m}) + \left( \bar{p}_0 \bar{m} + \Delta t\langle \bar{p}_1 , \bar{m} \bar{v} \rangle\right) -\left( \bar{p}_0 {m} +\Delta t \langle \bar{p}_1 , {m} {v} \rangle\right)+ \Delta t \frac{\alpha}{2} \|v - \bar{v}\|^2 m \\
     & \qquad \qquad + \sum_{x\in S} g(x) (m-\bar{m})(T,x)\\ 
    & \qquad =  \sum_{t\in\mathcal{T}} \sum_{x\in S}\bar{u}(t+1,x) \delta(t,x) +   \Delta t  \sum_{t\in\mathcal{T}} \sum_{x\in S}\frac{\alpha}{2}\|v - \bar{v}\|^2 m (t,x),
\end{split}
\end{equation*}
where the last equality was obtained with \eqref{eq:fund1} and \eqref{eq:fund2}.

\smallskip

\noindent \textbf{Step 2.} Let us prove that
\begin{equation*}
    \tilde{J}_{\bar{m}}(v,m) - \tilde{J}_{\bar{m}}(\bar{v},\bar{m}) \leq  \sum_{t\in\mathcal{T}} \sum_{x\in S} {u}(t+1,x) \delta(t,x) - \eta(m-\bar{m})(t,x) - \Delta t \frac{\alpha}{2}\|v - \bar{v}\|^2 \bar{m} (t,x) .
\end{equation*}
Since $v$  satisfies \eqref{eq:dmfgd}-(ii), by Fenchel's relation \cite[Cor.\@ 1.4.4]{JBHU}, we have 
\begin{equation*}
    \ell^{\bar{D}}[v]  = -\langle p_1, v \rangle - H^{\bar{D}}[p_1], \qquad -p_1\in \partial \ell^{\bar{D}}[v].
\end{equation*}
Then, by the $\alpha$-strong convexity of $\ell$ and the last equality, we have
\begin{equation*}
    \ell^{\bar{D}}[\bar{v}] \geq \ell^{\bar{D}}[v] + \langle -p_1, \bar{v} - v \rangle + \frac{\alpha}{2}\|\bar{v} - v\|^2 = -H^{\bar{D}}[p_1] - \langle p_1, \bar{v} \rangle + \frac{\alpha}{2}\|\bar{v} - v\|^2 .
\end{equation*}
Using the nonnegativity of $m$ and $\bar{m}$, we obtain that
\begin{equation*}
\begin{split} 
&    \left(\ell^{\bar{D}}[v]m - \ell^{\bar{D}}[\bar{v}]\bar{m} \right) \Delta t \leq \left(-H^{\bar{D}}[p_1](m-\bar{m}) - \langle p_1, mv - \bar{m}\bar{v} \rangle -\frac{\alpha}{2}\|\bar{v} - v\|^2 \bar{m}\right) \Delta t\\
& \quad \qquad    =  \left( {u} - {p}_0 -  \Delta t f(t,x,m(t) )  -\eta \right)(m-\bar{m}) - \Delta t \langle {p}_1, mv - \bar{m}\bar{v} \rangle -  \Delta t \frac{\alpha}{2}\|\bar{v} - v\|^2 \bar{m}\\
& \quad \qquad = {u} (m-\bar{m}) + \left( {p}_0 \bar{m} + \Delta t\langle {p}_1 , \bar{m} \bar{v} \rangle\right) -\left( {p}_0 {m} + \Delta t  \langle {p}_1 , {m} {v} \rangle\right)\\
 & \quad \qquad \qquad - \left(   \Delta t f(t,x,{m}(t))   + \eta\right) (m-\bar{m}) -  \Delta t \frac{\alpha}{2}\|\bar{v} - v\|^2 \bar{m}.
\end{split}
\end{equation*}
It follows that
\begin{equation*}
\begin{split}
    &\tilde{J}_{\bar{m}}(v,m) - \tilde{J}_{\bar{m}}(\bar{v},\bar{m}) \\
   & \quad = \Delta t  \sum_{t\in \mathcal{T}}\sum_{x\in S}\left(\ell^{\bar{D}}[v] m - \ell^{\bar{D}}[\bar{v}]\bar{m} + f(t,x,\bar{m}(t))(m-\bar{m})\right)(t,x)  +  \sum_{x\in S} g(x) (m-\bar{m})(T,x)\\ 
   & \quad \leq  \sum_{t\in\mathcal{T}}\sum_{x\in S}  {u} (m-\bar{m}) + \left( {p}_0 \bar{m} +\Delta t  \langle {p}_1 , \bar{m} \bar{v} \rangle\right) -\left( {p}_0 {m} + \Delta t \langle {p}_1 , {m} {v} \rangle\right) +  \sum_{x\in S} g(x) (m-\bar{m})(T,x)\\ 
    & \quad \quad -   \sum_{t\in\mathcal{T}}\sum_{x\in S} \Delta t \left( f(t,x,m(t)) - f(t,x,\bar{m}(t))\right)(m-\bar{m}) (t,x) \\
    & \quad \quad -  \sum_{t\in\mathcal{T}}\sum_{x\in S} \eta (m-\bar{m})(t,x) +\Delta t \frac{\alpha}{2}\|v - \bar{v}\|^2 \bar{m}(t,x) \\
   & \quad \leq  \sum_{t\in \mathcal{T}} \sum_{x\in S} {u}(t+1,x) \delta(t,x) - \eta(t,x)(m-\bar{m})(t,x) - \Delta t \frac{\alpha}{2}\|v - \bar{v}\|^2 \bar{m} (t,x) ,
\end{split}
\end{equation*}
where the last inequality is a consequence of \eqref{eq:fund3}, \eqref{eq:fund4}, and the monotonicity of $f$.

\section{Stability analysis for the theta-scheme}
\label{sec:stability}

We turn back to the stability analysis of the theta-scheme. It consists of two steps: the fundamental inequality, which is obtained by formulating \eqref{eq:dmfg} as a discrete MFG, and an energy estimate for the Kolmogorov equation.

From now on $\ell$, $H$, $g$, $m_0$, and $f$ are again to be understood according to the definitions given in \eqref{eq:grid1} and \eqref{eq:grid2}.

\subsection{Reformulation of the theta-scheme as a discrete MFG}

The goal of this subsection is to show the equivalence between the scheme \eqref{eq:theta_mfg2}
and a discrete MFG of the form \eqref{eq:dmfg}.
Given $D \in (0,\infty]$, define $\ell^D$ as in \eqref{eq:lD} and $H^D$ as in \eqref{eq:HD}. Note that for $D= \infty$, $H^D=H$.
Consider the following system, with unknown variables $u \in \R(\bar{\mathcal{T}} \times S)$, $v \in  \R^d(\mathcal{T} \times S)$, and $m \in \R(\bar{\mathcal{T}} \times S)$:
\begin{equation}\label{eq:theta_mfg}\tag{$\theta$-MFG$(D)$}
\left\{\begin{array}{cll}
\mathrm{(i)} \ & \begin{cases}
        \left( \text{Id} - \theta \sigma\Delta t \Delta_h\right) u(t+1/2)  =  u(t + 1) , \\
         u(t,x) =  \big[  - H^D [\nabla_h u ( \cdot +1/2 , \cdot ) ](t,x) + f(t,x,m(t)) \big]\Delta t \\
         \qquad \qquad    +  \big(\text{Id}+ (1-\theta )\sigma \Delta t \Delta_h  \big)u(t+1/2)(x), 
\end{cases} & \forall (t,x)\in \mathcal{T}\times S;\\
~\\
\mathrm{(ii)} \  & v(t,x) = -H_p^D[\nabla_h u(\cdot+1/2,\cdot)](t,x), &\forall (t,x)\in \mathcal{T}\times S;\\
~\\
\mathrm{(iii)} \ & \begin{cases}
          m(t+1/2) = \big(\text{Id}+ (1-\theta )\sigma \Delta t \Delta_h  \big) m(t)   - \Delta t \text{div}_h \big(v (t)m(t) \big) , \\
           \left( \text{Id} - \theta \sigma\Delta t \Delta_h\right) m(t+1) = m(t+1/2),
    \end{cases} &\forall t\in \mathcal{T};\\
~\\
\ \mathrm{(iv)} \ \ & m(0,x)=m_{0}(x), \quad u(T,x)=g(x), &\forall x\in S.
\end{array}\right.
\end{equation}
Multiplying the dynamic programming equation (i) and the Kolmogorov equation (iii) of \eqref{eq:theta_mfg2} by $\Delta t$, we easily see that \eqref{eq:theta_mfg2} is equivalent to \eqref{eq:theta_mfg} with $D=\infty$.

We recall here the definition of the matrix $B_1$ and introduce a new matrix $B_2$:
\begin{equation} \label{eq:b1b2}
    B_1 = \text{Id} - \theta \sigma\Delta t \Delta_h , \qquad B_2 =  (1-\theta )\sigma  \Delta_h.
\end{equation}
By Lemma \ref{lm:implicit}, the matrix $B_1$ is invertible.

We regard the variables $u$ and $m$ of the system \eqref{eq:theta_mfg} as elements of $\R(\bar{\mathcal{T}} \times S)$, since the auxiliary variables $u(t+1/2,\cdot)$ and $m(t+1/2, \cdot)$ are uniquely determined by $u(t+1,\cdot)$ and $m(t,\cdot)$. In the sequel, we will make use of the following convention: Given $u \in \R(\bar{\mathcal{T}} \times S)$, we denote
\begin{equation} \label{eq:convention}
u(t+1/2,\cdot)= B_1^{-1} u(t+1,\cdot), \quad \forall t \in \mathcal{T}.
\end{equation}

\begin{lem} \label{lem:equi1}
For any $D> 0$, the system  \eqref{eq:theta_mfg} is equivalent to the system \eqref{eq:dmfg} with running cost $\ell$, control bound $\bar{D}=D$, coupling function $f$, final cost $g$, initial distribution $m_0$ and with $\pi_0$ and $\pi_1$ defined by:
\begin{align}
     &\pi_0(t,x,y) = B_1^{-1}(y,x) + \Delta t (B_1^{-1}B_2)(y,x) ,\label{eq:Pi0} \\ &\pi_1(t,x,y) = \Big( \frac{B_1^{-1}(y,x+he_i)- B_1^{-1}(y,x-he_i)}{2h}\Big)_{i=1}^d. \label{eq:Pi1}
\end{align}
\end{lem}

\begin{proof}
We make use of the notations $p_0$, $p_1$, $q_0$ and $q_1$, defined as in \eqref{p}-\eqref{q}.
By the definition of $B_1$ and $B_2$, the implicit steps in equations (i) and (iii) are equivalent to
\begin{equation*}
    u(t+1/2) = B^{-1}_1 u(t+1)
    \quad \text{and} \quad m(t+1) = B^{-1}_1 m(t+1/2).
\end{equation*}
Next we verify the equivalences between each of the three equations of the two systems.

\smallskip

\noindent \textbf{Step 1.} Using the definition of $B_1$ and $\pi_1$, we have that
\begin{equation*}
    \begin{split}
        \nabla_h u(t+1/2,x)  & = \left(\frac{u(t+1/2,x+he_i) - u(t+1/2,x-he_i) }{2h} \right)_{i=1}^d\\
        & = \left(\frac{\sum_{y\in S} \Big( B_1^{-1}(x+he_i, y) - B_1^{-1}(x - he_i, y) \Big) u(t+1, y) }{2h}\right)_{i=1}^d \\
        & = \sum_{y\in S} \pi_1(t,x,y) u(t+1,y) = p_1(t,x).
    \end{split}
    \end{equation*}
    The equivalence of the feedback relations follows.
    
\smallskip    
    
\noindent \textbf{Step 2.} The dynamic programming equation is equivalent to 
\begin{equation*}
    u(t,x) = \Big[  - H^D [p_1 ](t,x) + f(t,x,m(t)) \Big]\Delta t + \big( \text{Id} + \Delta t B_2 \big)B^{-1}_1 u(t+1)(x).
\end{equation*}
Observe that
    \begin{equation*}
    \begin{split}
   &    \big( \text{Id} + \Delta t B_2 \big)B^{-1}_1 u(t+1)(x)  \\
       & \qquad = \sum_{y\in S} B_1^{-1}(x,y) u(t+1 , y)  + \Delta t \sum_{z \in S} \sum_{y\in S } B_2(x,z) B_1^{-1}(z, y)  u(t+1 , y) \\
  &  \qquad    = \sum_{y\in S} B_1^{-1}(y,x) u(t+1 , y)  + \Delta t \sum_{y\in S } \Big(\sum_{z \in S}  B_1^{-1}(y, z) B_2(z,x) \Big)  u(t+1 , y)\\
   &  \qquad   = \sum_{y \in S} \pi_0(t,x,y) u(t+1,y) = p_0(x),
    \end{split}
    \end{equation*}
    The equivalence with the dynamic programming equation of \eqref{eq:dmfg} follows.
    
\smallskip
    
\noindent \textbf{Step 3.} The Kolmogorov equation in \eqref{eq:theta_mfg} is equivalent to
\begin{equation*}
\begin{split}
    m(t+1,y) = {} & \sum_{x\in S}  \Big(B^{-1}_1(y,x)  + \Delta t (B_1^{-1}B_2)(y,x) \Big)m(t,x)\\
    & \quad - \Delta t \sum_{i=1}^d \sum_{x\in S} B_1^{-1}(y,x) \frac{v_i(t,x+he_i)m(t,x+he_i) - v_i(t,x-he_i)m(t,x-he_i)}{2h}\\
    = {} & \sum_{x\in S} \Big(B^{-1}(y,x)  + \Delta t (B_1^{-1}B_2)(y,x) \Big)m(t,x)\\
    & \quad + \Delta t  \sum_{x\in S} \sum_{i=1}^d \frac{B_1^{-1}(y,x+he_i)- B_1^{-1}(y,x-he_i)}{2h} v_i(t,x)m(t,x)\\
     ={} & \sum_{x\in S} \pi_0(t,x,y)m(t,x) + \Delta t \langle \pi_1(t,x,y), v(t,x)\rangle m(t,x) \\
     ={} & q_0(t,y) + \Delta t q_1[v](t,y).
\end{split}
\end{equation*}
The lemma is proved.
\end{proof}

\begin{lem} \label{lm:ass}
The maps $\ell$, $f$, and $g$ satisfy Assumption \ref{ass1} with the following constants:
\begin{equation*}
    \alpha = \alpha^c, \quad L_{\ell} = L_{\ell}^c, \quad L_{f} =L_f^c, \quad  L_f' = L_f^c h ^{-d/2},\quad
\text{and} \quad
    L_g = L_g^c.
\end{equation*}
\end{lem}

\begin{proof}
The Lipschitz-continuity of $\ell$, $f$, and $g$, and the strong convexity of $\ell$ are straightforward. For all $(t,x)\in \mathcal{T}\times S$ and $\mu_1,\, \mu_2 \in \mathcal{P}(S)$, we have
\begin{equation*}
    \begin{split}
        f(t,x,\mu_1) - f(t,x,\mu_2) &{} = \frac{1}{h^d} \int_{B_h(x)} f^c\Big(t\Delta t,x, \mathcal{R}_h(\mu_1)\Big) - f^c\Big(t\Delta t,x, \mathcal{R}_h(\mu_2)\Big) dy \\
        &{} \leq L_{f}^c \big\| \mathcal{R}_h(\mu_1) - \mathcal{R}_h(\mu_2) \big\|_{\mathbb{L}^2}  \\
        &{}=  L_{f}^c \Big(\sum_{x\in S}h^d \Big(\frac{\mu_1(x)}{h^d}-\frac{\mu_2(x)}{h^d}\Big)^2\Big)^{1/2} \\
        &{} = L_f^c h^{-d/2} \|\mu_1 -\mu_2\|_2.
    \end{split}
\end{equation*}
This proves the Lipschitz continuity of $f$ with respect to its third variable.
Let us consider again $\mu_1$, $\mu_2 \in \mathcal{P}(S)$ and $t \in \mathcal{T}$. We have
\begin{equation*}
    \begin{split}
         & \sum_{x\in S} (f(t,x,\mu_1)-f(t,x,\mu_2))(\mu_1(x)-\mu_2(x))\\
          & \qquad = \frac{1}{h^d} \sum_{x\in S} \int_{y\in B_h(x)} f^c(t\Delta t,y,\mathcal{R}_h(\mu_1) ) - f^c(t\Delta t,y,\mathcal{R}_h(\mu_2) ) dy ( \mu_1(x) - \mu_2(x)) \\
           & \qquad=   \sum_{x\in S} \int_{y\in B_h(x)} \Big(f^c(t\Delta t,y,\mathcal{R}_h(\mu_1) ) - f^c(t\Delta t,y,\mathcal{R}_h(\mu_2) ) \Big) \Big( \frac{\mu_1(x)}{h^d} - \frac{\mu_2(x)}{h^d} \Big) dy  \\
           & \qquad=  \sum_{x\in S} \int_{y\in B_h(x)} \Big( f^c(t\Delta t,y,\mathcal{R}_h(\mu_1) ) - f^c(t\Delta t,y,\mathcal{R}_h(\mu_2) ) \Big) \Big(\mathcal{R}_h(\mu_1)(y) - \mathcal{R}_h(\mu_2)(y) \Big) dy \\
          & \qquad =   \int_{\mathbb{T}^d} \Big(f^c(t\Delta t,y,\mathcal{R}_h(\mu_1) ) - f^c(t\Delta t,y,\mathcal{R}_h(\mu_2) ) \Big) \Big(\mathcal{R}_h(\mu_1)(y) - \mathcal{R}_h(\mu_2)(y)  \Big) dy \geq 0.
    \end{split}
\end{equation*}
This proves the monotonicity assumption. The lemma is proved.
\end{proof}

\begin{lem}[Lipschitz continuity]\label{lm:lip}
Let $D \in (0,\infty]$. Let $(m,u,v)$ be a solution to \eqref{eq:theta_mfg}. Suppose that $(\Delta t, h)$ satisfies the condition \eqref{cond:CFL3}.
Then for all $t \in \bar{\mathcal{T}}$, $u(t,\cdot)$ and $u(t+1/2,\cdot)$ are $( L_g^c+L_f^c + L_{\ell}^c )$-Lipschitz continuous. Moreover, $\| v \|_{\infty,\infty} \leq M$, where $M$ was defined in \eqref{eq:cons_M}.
\end{lem}

\begin{proof}
We define, for any $t \in \bar{\mathcal{T}}$, $L_{t} =  L_g^c + \Delta t ( T-t)(L_{f}^c + L_{\ell}^c ).$
We prove by induction that for any $t \in \bar{\mathcal{T}}$, $u(t,\cdot)$ is $L_t$-Lipschitz continuous. 
The claim is obvious for $t= T$, by the terminal condition.
Suppose that $u(t+1,\cdot)$ is $L_{t+1}$-Lipschitz for some $t\in\mathcal{T}$.
The first equation in \eqref{eq:theta_mfg} is equivalent to the dynamic programming equation:
\begin{equation}\label{eq:dp_u}
    \begin{cases}
        \left( \text{Id} - \theta \sigma\Delta t \Delta_h\right) u(t+1/2)  =  u(t + 1) ; \\
         u(t,x) =   \Delta t \inf_{\omega} \left\{f(t,x,m(t)) + \ell^D(t,x,\omega) + \Big\langle  \omega, \nabla_h u(t+1/2,x)  \Big\rangle\right\} \\
     \qquad \qquad +\big( \text{Id}+ (1-\theta )\sigma \Delta t \Delta_h  \big)u(t+1/2)(x), \qquad \forall \; x \in S.
    \end{cases}
\end{equation}
By the third statement of Lemma \ref{lm:implicit}, we have that $ u(t+1/2,\cdot)$ is $L_{t+1}$-Lipschitz. Therefore, $\|\nabla_h u(t+1/2,\cdot)\|\leq \sqrt{d}(L_g^c+L_f^c+L_{\ell}^c)$.
Next let us take $y \in S$ and let us set $\omega_y= v(t,y)$. We have
\begin{equation*}
    \omega_y = \argmin_{\| \omega\| \leq D} \ \Big( \ell(t,y,\omega) + \big\langle  \omega, \nabla_h u(t+1/2,y)  \big\rangle \Big).
\end{equation*}
By inequality \eqref{eq:bound:Hp} of Lemma \ref{lem:hD_diff}, we have $\| \omega_y \| \leq M$. Let $r' = (1 - \theta) \sigma \Delta t/h^2$, then, for any $x$, we have
\begin{equation*}
    \begin{split}
        u(t,x) - u(t,y)  \leq {} &  \Big( f(t,x,m(t)) - f(t,z,m(t)) + \ell(t,x,\omega_y) - \ell(t,y,\omega_y) \Big)\Delta t \\
       & \quad + (1-2dr') \Big( u (t+1/2,x) - u (t+1/2,y)\Big)\\
       & \quad + \sum_{i=1}^d \Big( r' +\frac{(\omega_z)_i}{2h}\Big)\Big( u(t+1/2,x+he_i) -u(t+1/2,y+he_i) \Big) \\
       & \quad + \sum_{i=1}^d\Big( r' -\frac{(\omega_z)_i}{2h}\Big)\Big( u(t+1/2,x-he_i) -u(t+1/2,y-he_i) \Big).
    \end{split}
\end{equation*}
The coefficients $\big(1-2dr' \big)$,  $\big( r' +\frac{(\omega_z)_i}{2h} \big)$, and $\big( r' -\frac{(\omega_z)_i}{2h} \big)$ are positive by the condition \eqref{cond:CFL3}. Thus,
\begin{equation*}
     u(t,x) - u(t,z) \leq \|x-z\|  \left(\Delta t(L_f^c + L_{\ell}^c) + L_{t+1}\right)
     \leq L_{t}\|x-z\|,
\end{equation*}
where $L_t= L_g + \Delta t ( T-t)(L_{f} + L_{\ell})\leq L_g + L_f + L_{\ell}$. Since $x$ and $y$ are arbitrary, we deduce that $u(t,\cdot)$ is $L_t$-Lipschitz. 
The claim is proved. We have meanwhile established that $\| v \|_{\infty,\infty} \leq M$. The lemma is proved, since for any $t \in \bar{\mathcal{T}}$, $L_t
\leq L_g^c+L_f^c + L_{\ell}^c$.
\end{proof}

\begin{thm} \label{thm:equivalence}
Let the condition \eqref{cond:CFL3} hold true.
Then the discretized MFG system \eqref{eq:theta_mfg2} is equivalent to the system \eqref{eq:theta_mfg} with $D= M$, which is itself equivalent to a discrete MFG of the form \eqref{eq:dmfg}, satisfying Assumptions \ref{ass1} and \ref{ass3}, with $\pi_0$ and $\pi_1$ defined by \eqref{eq:Pi0} and \eqref{eq:Pi1} and with control bound $\bar{D}=M$. As a consequence, \eqref{eq:theta_mfg2} has a unique solution.
\end{thm}

\begin{proof}
By construction, \eqref{eq:theta2} is equivalent to \eqref{eq:theta_mfg} with $D= \infty$. As a direct consequence of Corollary \ref{coro:equi_H} and Lemma \ref{lm:lip}, the system \eqref{eq:theta_mfg} with $D=\infty$ is equivalent to \eqref{eq:theta_mfg} with $D=M$. We already know that \eqref{eq:theta_mfg} is equivalent to \eqref{eq:dmfg}, by Lemma \ref{lem:equi1} and that Assumption \ref{ass1} is satisfied, by Lemma \ref{lm:ass}.
It remains to verify that \eqref{eq:theta_mfg} satisfy Assumption \ref{ass3}, for $\bar{D}= M$. To do this, we need to verify that for any $v \in \R^d(\mathcal{T} \times S) $ with $\| v \|_{\infty,\infty} \leq D$,  $\pi[v]$ is a transition process, by Lemma \ref{lem:trivial}. By Remark \ref{rem:charac_process}, this is equivalent to show that
for any $t \in \mathcal{T}$, for any $m(t) \in \mathcal{P}(S)$, for any $v \in \R^d(\mathcal{T}\times S)$ such that $\| v \|_{\infty,\infty} \leq M$, we have $m(t+1) \in \mathcal{P}(S)$, where $m(t+1)$ is defined by equation \eqref{eq:theta_mfg}-(iii). We conclude that $m(t+1)\in \mathcal{P}(S)$ from Remark \ref{rem:fp}.
\end{proof}

\subsection{Energy estimate for the discrete FP equation}

We investigate here the $\ell^2$-stability of the discrete Fokker-Planck equation of the theta-scheme. To this aim we consider the following perturbed equation:
  \begin{equation}\label{eq:theta2}
       \begin{cases}
       \begin{array}{rl}
       \left(\text{Id} - \theta \sigma\Delta t \Delta_h\right) \mu(t+1)  = & \big( \text{Id}+ (1-\theta )\sigma \Delta t \Delta_h  \big) \mu(t) - \Delta t \text{div}_h \big(v (t)\mu(t) \big)  \\
       &  - \Delta t \text{div}_h \big(\delta_v (t) \big) + \Delta t \delta(t), \\[0.5em]
      \mu(0)   = & \mu_0 ,
      \end{array}
   \end{cases} \\
   \end{equation}
where $ \delta_v\in \mathbb{R}^d(S), \delta \in \mathbb{R}(S)$. 
Note that we have no sign condition on $\mu$. 
The first error term $ - \Delta t \text{div}_h \big(\delta_v (t) \big) $ represents a perturbation in the form of a discrete divergence and $\Delta t \delta (t)$ is another general perturbation term.

\begin{prop}[Energy inequality]\label{prop:energy} Let $\theta > 1/2$ and $\mu$ be a solution of \eqref{eq:theta2}. 
Let $v \in \R^d(\mathcal{T} \times S)$ be such that $\| v \|_{\infty,\infty} \leq M$. Then, there exists some constant $c$ independent of $h$ and $\Delta t$ such that
\begin{equation}\label{eq:energy}
           \max_{t\in\tilde{\mathcal{T}}} \big\|\mu(t)\big\|_2^2  \leq c \left(  \big\|\mu_0 \big\|_2^2 +  (1-\theta) \sigma  \big\|\nabla^{+}_h \mu_0 \big\|_2^2 + \sum_{\tau\in\mathcal{T}} \Delta t \left(   \big\| \delta_v(\tau)\big\|^2_2 + \big\|\delta(\tau)\big\|_{2}^2 \right) \right).
\end{equation}
\end{prop}

\begin{proof}
Recall that the forward discrete gradient was defined in \eqref{eq:forward_grad}.
 Computing the scalar product with $\mu(t+1)$ of both sides of \eqref{eq:theta2} and applying the integration by parts formulas \eqref{eq:int_by_part1}-\eqref{eq:int_by_part2}, we obtain that
\begin{equation}\label{eq:dmu2}
\begin{split}
        \big\langle \mu (t+1) - \mu (t), \mu(t+1) \big\rangle + \theta \sigma \Delta t \beta_1
       = & (1-\theta) \sigma \Delta t \beta_2  + \Delta t\big( \gamma_1 + \gamma_2 + \gamma_3 \big),
\end{split}
\end{equation}
where 
\begin{align*}
       \beta_1 & {} = - \big\langle \mu(t+1), \Delta_h \mu(t+1) \big\rangle  =  \big\|\nabla^{+}_h \mu(t+1) \big\|_2^2 , \\[0.5em]
       \beta_2 & {} = \big\langle \mu(t+1), \Delta_h \mu(t) \big\rangle   = - \big\langle \nabla^{+}_h \mu(t+1) , \nabla^{+}_h \mu(t)\big\rangle, \\[0.5em]
       \gamma_1 & {}= - \big\langle \text{div}_h \big(\mu(t)v (t) \big), \mu(t+1)\big\rangle  = \sum_{x\in S}  \big\langle  \nabla_h \mu(t+1,x)  , \mu v(t,x)\big\rangle ,\\
       \gamma_2 & {}= - \big\langle \text{div}_h \big(\delta_v (t) \big), \mu(t+1)\big\rangle  = \sum_{x\in S}  \big\langle  \nabla_h \mu(t+1,x)  ,  \delta_v(t,x)\big\rangle, \\
       \gamma_3 & {} =   \big\langle \delta(t), \mu(t+1) \big\rangle.
\end{align*}
Using Young's inequality, it is easy to prove that
\begin{equation*}
       \Big\langle \mu (t+1) - \mu (t), \mu(t+1)\Big\rangle \geq  \frac{1}{2}\Big(\|\mu(t+1)\|_2^2 - \|\mu(t)\|_2^2 \Big).
\end{equation*}
Combining the above inequality with \eqref{eq:dmu2}, we obtain that
 \begin{equation}\label{eq:dmu3}
       \begin{split}
           &\frac{1}{2}\Big(\|\mu(t+1)\|_2^2 - \|\mu(t)\|_2^2 \Big)  + \theta \sigma \Delta t \big\|\nabla^{+}_h \mu(t+1) \big\|_2^2 \\[0.5em]
           & \qquad \quad \leq  - (1-\theta) \sigma \Delta t \, \big\langle \nabla^{+}_h \mu(t+1) , \nabla^{+}_h \mu(t)\big\rangle  + \Delta t \big( \gamma_1 + \gamma_2 +\gamma_3 \big).
       \end{split}
   \end{equation}
   Applying Young's inequality to the right-hand side of \eqref{eq:dmu3} and using inequality \eqref{eq:dmu+}, we obtain that for all positive numbers $\alpha_0, \alpha_1,\alpha_2$ and $\alpha_3$, we have
\begin{align*}
      - \big\langle \nabla^{+}_h \mu(t+1) , \nabla^{+}_h \mu(t)\big\rangle & \leq \frac{\alpha_0}{2} \big\|\nabla^{+}_h \mu(t+1) \big\|_2^2 + \frac{1}{2\alpha_0} \big\|\nabla^{+}_h \mu(t) \big\|_2^2 ;\\
      \gamma_1  & \leq \frac{\alpha_1}{2} \big\|\nabla_h ^{+}{\mu }(t+1) \big\|^2_2 +  \frac{M^2}{2\alpha_1}  \big\|\mu (t) \big\|^2_2;\\
       \gamma_2
       & \leq \frac{\alpha_2}{2}\big\|\nabla^{+}_h{\mu}(t+1)\big\|^2_2  +  \frac{1}{2\alpha_2}  \big\|\delta_v(t)\big\|^2_2;\\
      \gamma_3 & \leq \frac{\alpha_3}{2} \big\|\delta(t) \big\|_{2}^2 + \frac{1}{2\alpha_3} \big\|\mu(t+1) \big\|^2_2.
\end{align*}
Taking $\alpha_0 = 1$, $\alpha_1 = \alpha_2 = \sigma(2\theta -1) > 0$, and $\alpha_3 = 1$, we have
\begin{equation*}
       \begin{split}
           & (1-\Delta t) \big\| \mu(t+1) \big\|_2^2 - (1-\Delta t) \big\| \mu(t) \big\|_2^2 + (1-\theta) \sigma \Delta t \Big( \big\|\nabla^{+}_h \mu(t+1) \big\|_2^2 - \big\|\nabla^{+}_h \mu(t) \big\|_2^2   \Big) \\
       & \qquad \leq \Delta t \Big( c_1 \big\|\mu (t) \big\|^2_2 + c_2 \big\| \delta_v(t) \big\|^2_2 + \big\|\delta(t) \big\|_{2}^2 \Big),
\end{split}
\end{equation*}
where $c_1 = 1 + \frac{M^2}{\sigma(2\theta - 1)} $ and $c_2 = \frac{1}{\sigma(2\theta - 1)}  $.
Summing the above equation over $t$, it follows that
    \begin{equation*}
          (1-\Delta t) \big\|\mu(t+1)\big\|_2^2 
       \leq   \Delta t \sum_{\tau=0}^t\Big( c_1 \big\|\mu (\tau) \big\|^2_2 + c_2\big\|\delta_v(\tau)\big\|^2_2 + \big\|\delta(\tau) \big\|_{2}^2 \Big) + c_3,
   \end{equation*}
   where $c_3 = (1-\Delta t) \|\mu_0\|_2^2 +  (1-\theta) \sigma \Delta t \|\nabla^{+}_h \mu_0 \|_2^2$. Since $1-\Delta t \geq 1/2$, 
   by the discrete Gronwall inequality \cite{clark1987}, there exists some constant $c$ independent of $(\Delta t, h)$ such that
   \begin{align*}
       \max_{t\in\tilde{\mathcal{T}}} \big\|\mu(t)\big\|_2^2 & \leq c\left(  \big\|\mu_0 \big\|_2^2 +  (1-\theta) \sigma  \big\|\nabla^{+}_h \mu_0 \big\|_2^2 + \sum_{\tau\in\mathcal{T}} \Delta t \left(   \big\| \delta_v(\tau)\big\|^2_2 + \big\|\delta(\tau)\big\|_{2}^2 \right) \right).
   \end{align*}
The proposition is proved.
\end{proof}

\begin{rem}
 By taking $ \alpha_1 = \alpha_2 < \sigma(2\theta -1)$ in the proof, we can get a refined energy estimate with an additional term $ \sum_{t\in \tilde{\mathcal{T}}} \Delta t \big\|\nabla^{+}_h \mu(t) \big\|_2^2 $ on the left-hand side of \eqref{eq:energy}. This refined energy estimate is consistent with the continuous case \cite[Thm.\@ 2.1]{ladyvzenskaja1988linear}.
\end{rem}

\section{Consistency analysis of the theta-scheme} 

This section is dedicated to the consistency analysis of the theta-scheme and to the proof of Theorem \ref{thm:main}. To alleviate the proofs, we will make use of the big $\mathcal{O}$ notation: Given $f_1, f_2 \in \mathbb{R}^n(\mathcal{T}\times S)$ and $ \gamma > 0$, the notation $f_1  - f_2  = \mathcal{O}(h^{\gamma})$ (or $f_1 = f_2 + \mathcal{O}(h^{\gamma})$) means that there exists some constant $C$ independent of $h$ and $\Delta t$ such that
$ \|f_1 - f_2\|_{\infty,\infty}  \leq C h^{\gamma}$.
In particular, $f_1 = \mathcal{O}(h^{\gamma})$ means that $\|f_1\|_{\infty,\infty} \leq Ch^{\gamma}$.

All along the section, \eqref{cond:CFL3} is supposed to be satisfied. Therefore, we have $\Delta t = \mathcal{O}(h^2)$.

\subsection{Consistency error}

Let us recall that $(u^{*},v^{*},m^{*})$ is the unique solution to the continuous system \eqref{eq:mfg}. 
{The restriction of $(u^{*},m^{*})$ on the grid, denoted by $(u^{*}_h,v^{*}_h)$, is defined as in Theorem \ref{thm:main}.  Making use of the convention \eqref{eq:convention}, we define $v^{*}_h \in \R^d(\mathcal{T} \times S)$ by
\begin{equation}\label{eq:grid3}
    v^{*}_h(t,x) = -H_p^M[\nabla_h u^{*}_h(\cdot+1/2,\cdot)](t,x).
\end{equation} 
Then, $(u^{*}_h,v^{*}_h,m^{*}_h)$ can be considered as a solution of the perturbed discrete mean field game \eqref{eq:dmfgd} with perturbation terms $\eta$ and $\delta$ specified later in Lemma \ref{lm:consist+}.
}

\begin{lem} \label{lem:reg_sol}
For any $t \in \bar{\mathcal{T}}$, $u^{*}_h(t,\cdot)$ is $(L_{\ell}^c+L_{f}^c + L_g^c)$-Lipschitz continuous. For any $t \in \mathcal{T}$, $u^{*}_h(t+1/2,.)$ is also $(L_{\ell}^c+L_{f}^c + L_g^c)$-Lipschitz continuous. Moreover, $ \| v^{*}_h \|_{\infty} \leq M$ and
\begin{equation*}
\begin{split}
H^M[\nabla_h u^{*}_h(\cdot+1/2,\cdot)](t,x)
= {} &
H^c[\nabla_h u^{*}_h(\cdot+1/2,\cdot)](t,x) \\
H_p^M[\nabla_h u^{*}_h(\cdot+1/2,\cdot)](t,x)
= {} &
H_p^c[\nabla_h u^{*}_h(\cdot+1/2,\cdot)](t,x).
\end{split}
\end{equation*}
\end{lem}

\begin{proof}
It can be proved, with similar ideas to those of the proof of Lemma \ref{lm:lip}, that $u^*(t,\cdot)$ is $(L_{\ell}^c+L_{f}^c + L_g^c)$-Lipschitz continuous, for any $t \in [0,1]$. The first claim of the lemma follows immediately. The other claims can be shown with the same arguments as those of the proof of Lemma \ref{lm:lip}.
\end{proof}

Below we state (without proof) elementary consistency estimates, all directly deduced from Assumption \ref{ass:sol+}:
\begin{equation} \label{eq:consistency}
\begin{array}{rlrl}
\frac{u^*(t+\Delta t,x)-u^*(t,x)}{\Delta t} - \frac{\partial u^*(t,x)}{\partial t} & \!\!\! = \mathcal{O}(\Delta t^{r/2}), & \frac{m^*(t+\Delta t,x)-m^*(t,x)}{\Delta t} - \frac{\partial m^*(t,x)}{\partial t} & \!\!\! = \mathcal{O}(\Delta t^{r/2}), \\[0.5em]
\nabla_h u^*(t,x)-\nabla u^*(t,x)& \!\!\! =
\mathcal{O}(h^{1+r}), \qquad &
\mathrm{div}_h (m^{*}v^{*})(t,x)
- \mathrm{div}(m^{*}v^{*})(t,x)
& \!\!\! = \mathcal{O}(h^r), \\[0.5em]
\Delta_h u^*(t,x)- \Delta u^*(t,x) & \!\!\! =
\mathcal{O}(h^r), &
\Delta_h m^*(t,x) - \Delta m^*(t,x)
& \!\!\! = \mathcal{O}(h^r).
\end{array}
\end{equation}
We also observe that the discrete differential operators commute with integrals. For example,
\begin{equation} \label{eq:commutation}
\begin{split}
\Delta_h (\mathcal{I}_h(m^{*}))(t,x)
= {} & \frac{1}{h^2} \sum_{i=1}^d \int_{B_h(x)} \Big( m^{*}(t,y+he_i) + m^{*}(t,y-he_i) - 2m^{*}(t,y) \Big) dy \\
= {} & \int_{B_h(x)} \Delta_h m^{*}(t,y) dy
= \mathcal{I}_h( \Delta_h m^{*})(t,x).
\end{split}
\end{equation}
In the following three lemmas, we investigate the consistency errors associated with the coupling cost, the Hamiltonian, and the divergence term of the Fokker-Planck equation.

\begin{lem}\label{lm:f}
 For the global cost term, there holds: {for all $(t,x)\in \mathcal{T}\times S$,}
\begin{equation}\label{eq:df}
    f\big(t,x, m^{*}_h(t)\big) - f^c\big((t+1)\Delta t,x,m^{*}((t+1)\Delta t)\big) = \mathcal{O}(h) .
\end{equation}
\end{lem}

\begin{proof}
Since $m^{*}$ is Lipschitz continuous in time, uniformly in $x$, we have that
\begin{equation*}
\|m^{*}((t+1)\Delta t) - m^{*}(t\Delta t)\|_{\mathbb{L}^2} = \mathcal{O}(\Delta t).
\end{equation*}
Then the Lipschitz continuity of $f^c$ with respect to $t$ and $m$ implies that 
\begin{equation*}
    f^c\big((t+1)\Delta t,x,m^{*}((t+1)\Delta t)\big)- f^c\big(t\Delta t,x,m^{*}(t\Delta t)\big) = \mathcal{O}(\Delta t).
\end{equation*}
Using the definition of $f$ (provided in \eqref{eq:grid2}) and the Lipschitz continuity of $f^c$, we have
\begin{equation*}
\begin{split}
     & \big| f(t,x, m^{*}_h(t) ) - f^c(t\Delta t,x, m^{*}(t\Delta t)) \big| \\
  & \qquad   = \Big| \frac{1}{h^d}\int_{ B_h(x)} \Big( f^c (t\Delta t,y,\mathcal{R}_h \mathcal{I}_h(m^{*}(t\Delta t)) )  -   f^c(t\Delta t,x, m^{*}(t\Delta t)) \Big) dy \Big|\\
    & \qquad  \leq L_f^c \Big( \sqrt{d} h + \big\| \mathcal{R}_h \mathcal{I}_h(m^{*}(t\Delta t))  - m^{*}(t\Delta t)\big\|_{\mathbb{L}^2}\Big).
\end{split}
\end{equation*}
Then we estimate $ \big\| \mathcal{R}_h \mathcal{I}_h(m^{*}(t\Delta t))  - m^{*}(t\Delta t)\big\|_{\mathbb{L}^2}$ as follows:
\begin{equation*}
\begin{split}
   \big\| \mathcal{R}_h \mathcal{I}_h(m^{*}(t\Delta t))  - m^{*}(t\Delta t)\big\|_{\mathbb{L}^2} & = \Big(\sum_{x\in S} \int_{y\in B_h(x)} \Big| \frac{\mathcal{I}_h(m^{*}(t\Delta t))(x)}{h^d} - m^{*}(t\Delta t,y) \Big|^2 dy \Big)^{1/2}\\
    & \leq \Big(\sum_{x\in S} \int_{y\in B_h(x)} \int_{z\in B_h(x)} \frac{| m^{*}(t\Delta t,z) - m^{*}(t\Delta t,y)|^2}{h^d} dzdy \Big)^{1/2}\\
    & =\Big( \sum_{x\in S} \int_{y\in B_h(x)} \int_{z\in B_h(x)} \frac{\mathcal{O}(h^2)}{h^d} dzdy \Big)^{1/2} = \mathcal{O}(h),
\end{split}
\end{equation*}
where the second line is a consequence of Jensen's inequality. 
The lemma is proved.
\end{proof}

\begin{lem}\label{lm:H+}
It holds: {for all $(t,x)\in \mathcal{T}\times S$,}
\begin{equation}\label{eq:H_error+}
H^M\Big(t,x,\nabla_h u^{*}_h(t+1/2,x)\Big) - H^c\Big((t+1)\Delta t,x,\nabla u^{*}((t+1)\Delta t, x)\Big) =  \mathcal{O}(h^{1+r}).
\end{equation}
Moreover,
\begin{equation} \label{eq:H_error+2}
\begin{split}
& v^{*}((t+1)\Delta t,x)-v^{*}_h(t,x) \\
& \qquad  =H_p^M\Big(t,x,\nabla_h u^{*}_h(t+1/2,x)\Big) - H_p^c\Big((t+1)\Delta t,x,\nabla u^{*}((t+1)\Delta t, x)\Big) =  \mathcal{O}(h^{1+r}).
\end{split}
\end{equation}
\end{lem}

\begin{proof}
By Lemma \ref{lem:reg_sol}, 
we have $\| \nabla_h u^{*}_h(t+1/2,x) \| \leq C$ and $\| \nabla u^*((t+1)\Delta, x) \| \leq C$, where $C = \sqrt{d} (L_{\ell}^c + L_f^c + L_g^c)$. Let $\Omega$ denote the closed ball of radius $\Omega$.
Since $H^c$ is uniformly Lipschitz with respect to $t$ and continuously differentiable with respect to $p$ (see Lemma \ref{lm:l}), we deduce that $H^c(\cdot,x,\cdot)$ is Lipschitz continuous on $[0,T] \times \Omega$, uniformly in $x$. Let $L_H$ denote the corresponding modulus. Then,
\begin{equation*}
\begin{split}
     & \Big |H^M \Big(t,x,\nabla_h u^{*}_h(t+1/2,x)\Big) - H^c\Big((t+1)\Delta t,x,\nabla u^{*}((t+1)\Delta t, x)\Big)\Big| \\
     & \qquad \leq L_H \| \nabla_h u^{*}_h(t+1/2,x) - \nabla u^{*}((t+1)\Delta t, x) \| + \mathcal{O}( \Delta t).
\end{split}
\end{equation*}
It is easy to deduce from the regularity of $u^*$ (Assumption \ref{ass:sol+}) that $\Delta_h u^{*}_h(t+1,\cdot)$ is H\"olderian with exponent $r$. Then, using the fourth statement of Lemma \ref{lm:implicit} and the consistency estimate \eqref{eq:consistency}, we obtain that
\begin{equation*}
    \nabla_h u^{*}_h(t+1/2,x) = \nabla_h u^{*}_h(t+1,x) + \mathcal{O}(\Delta t h^{r-1}) = \nabla u^{*}((t+1)\Delta t,x) + \mathcal{O}(\Delta t h^{r-1} + h^{1+r}).
\end{equation*}
The estimate \eqref{eq:H_error+} follows and estimate \eqref{eq:H_error+2} can be proved similarly.
\end{proof}

\begin{lem}\label{lm:v+}
For the divergence term, there holds: {for all $(t,x)\in \mathcal{T}\times S$,}
\begin{align}
     \textnormal{div}_h (v^{*}_h m^{*}_h (t,x)) - \int_{B_h(x)}\textnormal{div} (v^{*} m^{*})((t+1)\Delta t,y) dy = \mathcal{O} (h^{r+d}) +\textnormal{div}_h(\epsilon_1) ; \label{eq:v_error+}\\
     \textnormal{div}_h (v^{*}_h m^{*}_h (t,x)) - \int_{B_h(x)}\textnormal{div} (v^{*} m^{*}) (t\Delta t,y)dy = \mathcal{O} (h^{r+d}) +\textnormal{div}_h( \epsilon_2), \label{eq:v_error_2+}
\end{align}
where $\epsilon_1 = \mathcal{O}(h^{1+r+d})$, and $\epsilon_2 =  \mathcal{O}(h^{2r+d})$.
\end{lem}

\begin{proof}
In order to prove \eqref{eq:v_error+}, let us decompose $v^{*}_h m^{*}_h$ as the sum of three terms, $\gamma_1$, $\gamma_2$, and $\gamma_3$, defined by
\begin{align*}
    \gamma_1(t,x) & {} = \Big(v^{*}_h(t,x)-v^{*}((t+1)\Delta t, x)\Big)m^{*}_h(t,x); \\
    \gamma_2(t,x) & {} = v^{*}((t+1)\Delta t, x) \Big(m^{*}_h(t,x)- m^{*}_h(t+1,x) \Big); \\
    \gamma_3(t,x) & {} = v^{*}((t+1)\Delta t, x) m^{*}_h(t+1,x).
\end{align*}

\smallskip
\noindent
\textbf{Step 1:} Estimation of $\gamma_1$. Since $m^{*}_h(t,x)= \mathcal{O}(h^d)$, we directly obtain with Lemma \ref{lm:H+} that
$\gamma_1(t,x)= \mathcal{O}(h^{1+r+d})$.
    
\smallskip
\noindent
\textbf{Step 2:} Estimation of $\gamma_2$. By the definition of $m^{*}_h$, we have 
\begin{equation*}
m^{*}_h(t,x) - m^{*}_h(t+1,x) = \int_{B_h(x)} \Big( m^{*}(t\Delta t, y) - m^{*}((t+1)\Delta t, y) \Big)  dy = \mathcal{O}(\Delta t h^d).
\end{equation*}
Then $\gamma_2 = \mathcal{O}(h^{2+d})$, since $v^{*}$ is uniformly bounded.
      
\smallskip
\noindent
\textbf{Step 3:} Estimation of $\textnormal{div}_h \gamma_3$. Using the definitions of $\gamma_3$ and $m^{*}_h$, we obtain that
      \begin{equation*}
      \begin{split}
\textnormal{div}_h (\gamma_3)(t,x)
          & {} =  \int_{B_h(0)}\textnormal{div}_h \Big(v^{*}((t+1)\Delta t, \cdot) m^{*}((t+1)\Delta t, \cdot + y)\Big)(x) \, dy\\
           & {} =  \int_{B_h(0)}\textnormal{div} \Big(v^{*}((t+1)\Delta t, \cdot ) m^{*}((t+1)\Delta t, \cdot + y) \Big)(x) \, dy + \mathcal{O}(h^{r+d})\\
          & {} = \int_{B_h(0)}\textnormal{div} \Big(v^{*}((t+1)\Delta t, \cdot + y) m^{*}((t+1)\Delta t, \cdot + y) \Big)(x) \, dy + \mathcal{O}(h^{r+d})\\
           & {} =  \int_{B_h(x)}\textnormal{div} (v^{*} m^{*} ) ((t+1)\Delta t)(y) \, dy + \mathcal{O} (h^{r+d}).
      \end{split}
      \end{equation*}
The second equality follows from the fact that $(v^{*}m^{*})((t+1)\Delta t,\cdot+y)\in \mathcal{C}^{1+r}(\mathbb{T}^d)$. For the third one, we use that $v^*$ and $D_x v$ are H\"olderian with exponent $r$. Then, the estimate \eqref{eq:v_error+} holds true.
      
\smallskip
\noindent
\textbf{Step 4:} Proof of \eqref{eq:v_error_2+}. Since $v^{*}m^{*}(t \Delta t,\cdot)$ and $v^{*}m^{*}((t+1)\Delta t,\cdot)$ lie in $\mathcal{C}^{1+r}(\mathbb{T}^d)$, we first have that
\begin{align*}
    &\textnormal{div}(v ^{*}m^{*}) (t\Delta t,y) - \textnormal{div}_h (v ^{*}m^{*}) (t\Delta t,y)   =\mathcal{O}(h^r);\\
   & \textnormal{div}(v ^{*}m^{*} ((t+1)\Delta t,y)) - \textnormal{div}_h (v ^{*}m^{*}) ((t+1)\Delta t,y)  =\mathcal{O}(h^r).
\end{align*}
Since $ v^{*}m^{*}(\cdot,y)\in \mathcal{C}^{r}([0,1])$, we have
\begin{equation*}
    v^{*} m^{*} ((t+1)\Delta t,y) - v^{*} m^{*} (t\Delta t,y) = \mathcal{O}(\Delta t^{r})  = \mathcal{O}(h^{2r}).
\end{equation*}
Then we have
\begin{equation*}
    \begin{split}
      & \int_{B_h(x)} \textnormal{div}_h \Big( v^{*} m^{*} ((t+1)\Delta t,\cdot) - v^{*} m^{*} (t\Delta t,\cdot)(y) \Big) dy \\
      & \qquad  \quad = \textnormal{div}_h \Big(\int_{B_h(0)} \big( v^{*} m^{*} ((t+1)\Delta t,\cdot + y) - v^{*} m^{*} (t\Delta t, \cdot + y) \big) dy \Big)(x).
    \end{split}
\end{equation*}
The right-hand side is a discrete divergence of a term of order $\mathcal{O}(h^{2r+d})$. The estimate \eqref{eq:v_error_2+} follows.
\end{proof}

We are ready to derive a complete consistency estimate for the triplet $(u^{*}_h,v^{*}_h,m^{*}_h)$ defined at the beginning of the section.

\begin{lem}[Consistency error]\label{lm:consist+}
The triplet $(u^{*}_h,v^{*}_h,m^{*}_h)$ is a solution to the perturbed discrete mean field game \eqref{eq:dmfgd} with perturbation terms $\eta$ and $\delta$ satisfying
  \begin{equation*}
      \eta = \mathcal{O}(\Delta t h^r), \qquad \delta = \mathcal{O}(\Delta t h^{r+d}) + \Delta t\, \textnormal{div}_h(\epsilon_3), \qquad \text{where }\epsilon_3 = \mathcal{O}(h^{2r+d}).
  \end{equation*}
\end{lem}

\begin{proof}
\textbf{Step 1.} The perturbation term $\eta$ of the dynamic programming equation is defined by
\begin{equation*}
-\frac{u^{*}_h(t+1,x) - u^{*}_h(t,x)}{\Delta t} - \sigma \Delta_h u^{*}_h(t+1/2, x) +  H^M [\nabla_h u^{*}_h ( \cdot +1/2 , \cdot ) ](t,x) = f(t,x,m^{*}_h(t)) + \frac{\eta(t,x)}{\Delta t}.
\end{equation*}
The continuous HJB equation, satisfied by $u^*$, reads at time $(t+1) \Delta t$ as follows:
\begin{equation*}
     - \frac{\partial u^{*}((t+1)\Delta t, x)}{\partial t}  - \sigma \Delta u^{*}( (t+1) \Delta t, x) + H^c [\nabla u^{*}]((t+1)\Delta t, x) =  f^c\Big((t+1)\Delta t,x,m^{*}((t+1)\Delta t) \Big).
\end{equation*}
Then $\eta$ can be put in the form $\eta= \Delta t (r_1 + r_2 + r_3 + r_4)$, where
\begin{align*}
r_1(t,x) & {}= \frac{\partial u^{*}((t+1)\Delta t, x)}{\partial t} - \frac{u^{*}_h(t+1,x) - u^{*}_h(t,x)}{\Delta t} ;\\
    r_2(t,x) & {} = \sigma \Big( \Delta u^{*}( (t+1) \Delta t, x) - \Delta_h u^{*}_h(t+1/2, x) \Big); \\
    r_3(t,x) & {} = H^M [\nabla_h u^{*}_h ( \cdot +1/2 , \cdot ) ](t,x) - H^c [\nabla u^{*}]((t+1)\Delta t, x) ;\\
    r_4(t,x) & {} = f^c\Big((t+1)\Delta t,x,m^{*}((t+1)\Delta t) \Big) - f(t,x,m^{*}_h(t)). 
\end{align*}
By \eqref{eq:consistency}, we have $r_1 =\mathcal{O}(\Delta t^{r/2}) =\mathcal{O}(h^r)$.
 Since $u^{*}((t+1)\Delta t,\cdot ) \in \mathcal{C}^{2+r}(\mathbb{T}^d)$, it follows that $\Delta_h u^{*}_h(t+1,\cdot)$ is $r$-H\"older continuous. Using the fourth statement of Lemma \ref{lm:implicit} and \eqref{eq:consistency}, we obtain that
\begin{equation*}
    \Delta_h u^{*}_h(t+1/2, x)  = \Delta_h u^{*}_h(t+1, x) + \mathcal{O}(\Delta t h^{r-2}) = \Delta u^{*}( (t+1) \Delta t, x) + \mathcal{O}(\Delta th^{r-2}  + h^r).
\end{equation*}
Thus $r_2 = \mathcal{O}(h^r)$. Lemmas \ref{lm:f} and Lemma \ref{lm:H+} yield
$r_3 = \mathcal{O}(h^{1+r})$ and $r_4 = \mathcal{O}(h)$. It follows that $\eta(t,x) = \mathcal{O}(\Delta t h^r)$. 

\smallskip
\noindent
\textbf{Step 2.} The perturbation term of the discrete Fokker-Planck equation satisfies
\begin{equation*}
   \frac{m^{*}_h(t+1,x) - m^{*}_h(t,x)}{\Delta t} - \sigma \theta  \Delta_h m^{*}_h(t+1, x) - (1-\theta )\sigma \Delta_h m^{*}_h(t, x) + \text{div}_h(v^{*}_h m^{*}_h (t,x))  = \frac{\delta(t,x)}{\Delta t}. 
\end{equation*}
The Fokker-Planck equation, satisfied by $m^*$, writes as follows at times $t \Delta t$ and $(t+1)\Delta t$:
\begin{equation*}
\begin{split}
     & \frac{\partial m^{*}(t\Delta t, x)}{\partial t}  - \sigma \Delta m^{*}(t \Delta t, x) + \textnormal{div}(v^{*}m^{*}(t\Delta t, x)) =  0; \\
     & \frac{\partial m^{*}((t+1)\Delta t, x)}{\partial t}  - \sigma \Delta m^{*}( (t+1) \Delta t, x) + \textnormal{div}(v^{*}m^{*}((t+1)\Delta t, x))  =  0.
\end{split}
\end{equation*}
Let us integrate over $B_h(x)$ the convex combination of the last two equations:
\begin{equation*}
\begin{split}
   & (1-\theta) \int_{y\in B_h(x)}\frac{\partial m^{*}(t\Delta t , y)}{\partial t} - \sigma  \Delta m^{*}( t \Delta t, y) + \textnormal{div}(v^{*}m^{*}(t\Delta t, y)) dy \\
   &\qquad + \theta \int_{y\in B_h(x)}\frac{\partial m^{*}((t+1)\Delta t , y)}{\partial t}  - \sigma  \Delta m^{*}( (t+1) \Delta t, y)+\textnormal{div}(v^{*}m^{*}((t+1)\Delta t, y)) dy =  0.
\end{split}
\end{equation*}
Then $\delta= \Delta t (\bar{r}_1 + \bar{r}_2 +\bar{r}_3 + \tilde{r}_1 + \tilde{r}_2 + \tilde{r}_3)$, where
\begin{align*}
\bar{r}_1(t,x) & {}= \theta\Big( \frac{m^{*}_h(t+1,x) - m^{*}_h(t,x)}{\Delta t} - \int_{y\in B_h(x)}\frac{\partial m^{*}((t+1)\Delta t , y)}{\partial t}  dy \Big); \\
    \bar{r}_2(t,x) & {} = \sigma \theta \Big( \int_{y\in B_h(x)}\Delta m^{*}( (t+1) \Delta t, y) dy - \Delta_h m^{*}_h(t+1, x) \Big) ; \\
    \bar{r}_3(t,x) & {} = \theta \Big( \text{div}_h(v^{*}_h m^{*}_h (t,x)) - \int_{y\in B_h(x)}\textnormal{div}(v^{*}m^{*}((t+1)\Delta t, y)) dy \Big); \\
    \tilde{r}_1(t,x) & {} = (1-\theta)\Big( \frac{m^{*}_h(t+1,x) - m^{*}_h(t,x)}{\Delta t} - \int_{y\in B_h(x)}\frac{\partial m^{*}( t \Delta t , y)}{\partial t}  dy \Big); \\
     \tilde{r}_2(t,x) & {} = \sigma (1-\theta) \Big( \int_{y\in B_h(x)}\Delta m^{*}(t \Delta t, y) dy - \Delta_h m^{*}_h(t, x) \Big) ; \\
     \tilde{r}_3(t,x) & {}= (1-\theta) \Big( \text{div}_h(v^{*}_h m^{*}_h (t,x)) - \int_{y\in B_h(x)}\textnormal{div}(v^{*}m^{*}(t \Delta t, y)) dy \Big).
\end{align*}
Using the basic consistency estimates in \eqref{eq:consistency} and the commutation property shown in \eqref{eq:commutation}, we have $\bar{r}_1 =\mathcal{O}(\Delta t^{r/2} h^d) = \mathcal{O}(h^{r+d})$, $ \bar{r}_2 = \mathcal{O}(h^{r+d})$, $\tilde{r}_1 = \mathcal{O}(\Delta t^{r/2} h^d) = \mathcal{O}(h^{r+d})$, and $\tilde{r}_2 = \mathcal{O}( h^{r+d})$. Lemma \ref{lm:v+} shows that $\bar{r}_3 = \mathcal{O}(h^{r+d})+\theta \textnormal{div}_h \epsilon_1$ and $\tilde{r}_3 = \mathcal{O}(h^{r+d}) + (1-\theta)\textnormal{div}_h \epsilon_2$. Taking $\epsilon_3 = \theta \epsilon_1 + (1-\theta )\epsilon_2$, the conclusion follows.
\end{proof}

\subsection{Proof of Theorem \ref{thm:main}}\label{sec:proof}

All constants in the proof are independent of $\Delta t$ and $h$. The existence and uniqueness of the solution $(u_h,v_h,m_h)$ to the theta-scheme was established in Theorem \ref{thm:equivalence}.
The triplet $(u_h,v_h,m_h)$ is also the unique solution to \eqref{eq:dmfg} with control bound $M$. We proved in Lemma \ref{lm:consist+} that $(u^{*}_h,v^{*}_h,m^{*}_h)$ is a solution to \eqref{eq:dmfgd} with perturbation terms $\eta$ and $\delta$ estimated as follows:
\begin{equation*}
\eta = \mathcal{O}(\Delta t h^r), \qquad \delta = \mathcal{O}(\Delta t h^{r+d}) + \Delta t\textnormal{div}_h(\epsilon_3), \qquad \text{where }\epsilon_3 = \mathcal{O}(h^{2r+d}).
\end{equation*}

\smallskip
\noindent
\textbf{Step 1.} Using similar arguments to the ones of the proof of Theorem \ref{thm1} (see in particular estimate \eqref{eq:reg_u}), we easily obtain that
\begin{equation}\label{eq:delta_u+}
    \|u^{*}_h - u_h\|_{\infty,\infty}  \leq L_{f}^c \frac{\|m^{*}_h-m_h\|_{\infty,2} }{h^{d/2}} + \|\eta\|_{1,\infty}.
\end{equation}

\smallskip
\noindent
\textbf{Step 2.} Next we apply the fundamental inequality (Proposition \ref{prop:error_fund}) to $(u_h,v_h,m_h)$ and $(u^{*}_h,v^{*}_h,m^{*}_h)$. We obtain
\begin{equation} \label{eq:stab_hjb_eta}
      \frac{\Delta t \alpha}{2} \left\|\|v^{*}_h-v_h\|^2(m^{*}_h+m_h)\right\|_{1,1} \leq  \sum_{t\in\mathcal{T}} \sum_{x\in S} \Big( (u^{*}_h-u_h)(t+1,x) \delta(t,x) + (m_h-m^{*}_h) (t,x)\eta (t,x) \Big).
\end{equation}
Let us bound the right-hand side of the obtained inequality.
There exist two constants $C_0$ and $C_1$ such that
\begin{equation*}
    \begin{split}
       & \sum_{t\in\mathcal{T}} \sum_{x\in S} (u^{*}_h- u_h)(t+1,x) \delta(t,x) \\
      & \quad \leq  \sum_{t\in\mathcal{T}} \sum_{x\in S} \Big( \Delta t (u^{*}_h- u_h)(t+1,x) \textnormal{div}_h (\epsilon_3(t,x)) \Big) + \|u^{*}_h-u_h\|_{\infty,\infty} \|\delta -\Delta t\textnormal{div}_h (\epsilon_3) \|_{1,1}
        \\
       & \quad \leq  \sum_{t\in\mathcal{T}} \sum_{x\in S} \Big( - \Delta t \Big\langle \nabla_h (u^{*}_h-u_h) (t+1,x), \epsilon_3(t,x) \Big\rangle \Big) + C_0  \Big( L_{f}^c\frac{\|m^{*}_h-m_h\|_{\infty,2} }{h^{d/2}} + \|\eta\|_{1,\infty} \Big) h^r\\
        & \quad \leq C_1 \big(  h^{2r} + \|m^{*}_h-m_h\|_{\infty,2} h^{r-d/2} \big).
            \end{split}
\end{equation*}
The first inequality is a consequence of H\"older's inequality and the second one derives from the discrete integration by parts formula combined with inequality \eqref{eq:delta_u+}.
The third one follows from the Lipschitz continuity of $u$ and $u_h$.
By H\"older's inequality, there also exists a constant $C_2$ such that
\begin{equation*}
     \sum_{t\in\mathcal{T}} \sum_{x\in S} (m_h - m^{*}_h)(t,x) \eta(t,x) \leq   \|m^{*}_h - m_h\|_{\infty,2} \|\eta\|_{1,2} \leq C_2  \|m^{*}_h - m_h\|_{\infty,2} h^{r-d/2}.
\end{equation*}
Then, there exists a constant $C_3$ such that
\begin{equation}\label{eq:epsilon+}
    \epsilon \leq C_3 \big( h^{2r} + \|m^{*}_h-m_h\|_{\infty,2} h^{r-d/2} \big),
    \quad
    \text{where: } \epsilon =  \Delta t \left\|\|v^{*}_h-v_h\|^2m\right\|_{1,1}.
\end{equation}

\smallskip
\noindent
\textbf{Step 3.} We next find an upper bound of $\|m^{*}_h-m_h\|_{\infty,2}$ involving $\epsilon$, using the energy estimate established in Proposition \ref{prop:energy}.
Let $\mu = m^{*}_h - m_h$. Then $\mu$ satisfies the perturbed discrete Fokker-Planck equation defined in \eqref{eq:theta2}:
  \begin{equation*}
       \begin{cases}
       \begin{array}{rl}
       \left( \text{Id} - \theta \sigma\Delta t \Delta_h\right) \mu(t+1)  = & \! \! \! \big( \text{Id}+ (1-\theta )\sigma \Delta t\, \Delta_h  \big) \mu(t) - \Delta t \, \text{div}_h \big(v_h (t)\mu(t) \big)  \\
       &  - \Delta t \, \text{div}_h \big(\delta_v (t)  \big) + \Delta t\, \delta' (t), \\[0.2em] 
      \mu(0)   = & 0 ,
      \end{array}
   \end{cases} \\
   \end{equation*}
where \begin{equation*}
    \delta_v(t,x) = (v^{*}_h-v_h)m^{*}_h(t,x) - \epsilon_3(t,x) \qquad \text{ and } \qquad  \delta' = \mathcal{O}(h^{r+d}).
\end{equation*}
From Theorem \ref{thm:equivalence} we know that $\| v_h \|_{\infty,\infty} \leq M$. Thus, the energy inequality \eqref{eq:energy} implies that there exists a constant $C_4$ such that
\begin{equation*}
           \max_{t\in\tilde{\mathcal{T}}} \big\|\mu(t)\big\|_2^2  \leq C_4   \sum_{\tau\in\mathcal{T}} \Delta t \left(   \big\|  \delta_v(\tau) \big\|^2_2 + \big\| \delta'(\tau) \big\|_{2}^2 \right).
\end{equation*}
Applying inequality $(a+b)^2 \leq 2a^2+ 2b^2$ to $\|\delta_v(\tau)\|_2^2 $, there exists a constant $C_5$ such that
\begin{align*}
   &      \big\| \delta_v(\tau) \big\|^2_2 \leq 2 \big\|  (v^{*}_h-v_h) m^{*}_h (\tau)\big\|^2_2 + 2 \big\| \epsilon_3 (\tau) \big\|^2_2 \leq 2  \big\|  (v^{*}_h-v_h) m^{*}_h (\tau) \big\|^2_2 +C_5 h^{4r+d}, \\[0.5em]
 & \big\|\delta'(\tau) \big\|_{2}^2 \leq C_5 h^{2r+d}.
\end{align*}
Since $\|m^{*}_h\|_{\infty,\infty} = \mathcal{O}(h^d)$, there exists a constant $C_6$ such that
\begin{equation*}
   \sum_{\tau\in\mathcal{T}} \Delta t \left(   \big\|  \delta_v(\tau) \big\|^2_2 + \big\| \delta'(\tau) \big\|_{2}^2 \right) \leq C_6 h^d \Big( \epsilon + h^{2r}\Big).
\end{equation*}
Therefore, for some constant $C_7$,
\begin{equation} \label{eq:delta_m+}
    \|m^{*}_h-m_h\|_{\infty,2}^2 = \|\mu\|_{\infty,2}^2 \leq C_7 h^{d} \big( \epsilon + h^{2r} \big).
\end{equation}

\smallskip
\noindent
\textbf{Step 4.} Let us combine inequality \eqref{eq:epsilon+} with \eqref{eq:delta_m+}. We obtain that
\begin{equation*}
\begin{split}
     \|m^{*}_h-m_h\|_{\infty,2}^2 &\leq C_7(C_3+1)h^{2r+d}  + C_7C_3 \|m^{*}_h-m_h\|_{\infty,2} h^{r+ d/2} \\
     & \leq C_7(C_3+1)h^{2r+d} +  \frac{\|m^{*}_h-m_h\|_{\infty,2}^2 }{2} + \frac{C_7^2C_3^2}{2} h^{2r+d}.
\end{split}
\end{equation*}
Therefore, for some constant $C_8$,
\begin{equation}\label{eq:delta_m_2+}
    \|m^{*}_h-m_h\|_{\infty,2} \leq C_8 h^{r+d/2}.
\end{equation}
Applying H\"older's inequality to \eqref{eq:delta_m_2+} and using \eqref{eq:delta_u+}, we obtain the existence of a constant $C_9$ such that
\begin{equation*}
    \|u_h - u^{*}_h\|_{\infty,\infty} + \|m_h - m^{*}_h\|_{\infty,1} \leq C_9 h^{r}.
\end{equation*}
The conclusion follows.

\appendix

\section{Technical lemmas and proofs}\label{Appendix:A}

\begin{proof}[Proof of Lemma \ref{lm:int_by_part}] \label{proof:int_by_part}
We prove \eqref{eq:int_by_part1}:
\begin{equation*}
\begin{split}
     - \sum_{x\in S} \mu(x) \text{div}_h \omega (x) & = -\sum_{x\in S}\sum_{i=1}^d \mu(x) \frac{\omega_i(x+he_i) - \omega_i(x - he_i) }{2h}\\
     & = - \sum_{x\in S}\sum_{i=1}^d \omega_i(x) \frac{\mu(x-he_i) - \mu(x + he_i) }{2h}= \sum_{x\in S}  \left\langle  \nabla_h {\mu}(x)  , {\omega}(x)\right\rangle.
\end{split}
\end{equation*}
We prove \eqref{eq:int_by_part2}:
\begin{equation*}
\begin{split}
&          -  \sum_{x\in S} \nu (x) \Delta_h \mu (x)
         = - \sum_{x\in S}\sum_{y\in S} \nu(x) \Delta_h(x,y) \mu(y)\\
& \qquad          =  \frac{1}{h^2}\sum_{x\in S} \nu(x) \sum_{i=1}^d \big( 2\mu(x) - (\mu(x+he_i) + \mu(x-he_i)) \big) \\
& \qquad          = \frac{1}{h^2} \sum_{i=1}^d \Big( \sum_{x\in S} \mu(x+he_i) \big( \nu(x+he_i) - \nu(x) \big) - \sum_{x\in S} \mu(x-he_i) \big( \nu(x) - \nu(x- he_i) \big)  \Big) \\
& \qquad           = \sum_{x \in S} \left\langle \nabla^{+}_h \nu(x), \nabla^{+}_h \mu(x)  \right\rangle .
\end{split}
\end{equation*}
The lemma is proved.
\end{proof}

\begin{proof}[Proof of Lemma \ref{lem:easy_ineq}] \label{proof:easy_ineq}
Applying the inequality $(a+b)^2 \leq 2a^2 + 2b^2$, we obtain
     \begin{equation*}
     \begin{split}
         \|\nabla_h \mu\|_2 ^2 & = \frac{1}{4h^2} \sum_{i=1}^d \sum_{x\in S} \big(\mu(x+he_i)- \mu(x-he_i)\big)^2\\
         & \leq \frac{1}{2h^2} \sum_{i=1}^d \sum_{x\in S} \big(\mu(x+he_i)- \mu(x)\big)^2 + \big(\mu(x)- \mu(x - he_i)\big)^2 = \|\nabla_h^{+} \mu\|_2 ^2.
     \end{split}      
     \end{equation*}
     Inequality \eqref{eq:dmu+} follows.
\end{proof}

\begin{proof}[Proof of Lemma \ref{lm:implicit}] \label{proof:implicit}
Let $r = c\Delta t / h^2$. Consider the mapping  $\mathbb{S}_{X}(\mu) \colon \mathbb{R}(S) \rightarrow \mathbb{R}(S)$, defined by
    \begin{equation}\label{eq:map}
    \mathbb{S}_{X}(\mu)(x) = \frac{1}{1+2dr} \Bigg( r\sum_{j=1}^d \mu(x+he_j) + r\sum_{j=1}^d \mu(x-he_j) + X(x) \Bigg). 
    \end{equation}
    Then $Y$ is a solution to \eqref{eq:implicit} if and only if it is a fixed point of $\mathbb{S}_{X}$. For any $\mu_1$ and $\mu_2$ in $\mathbb{R}(S)$ and for any $ x \in S$,
    \begin{equation} \label{eq:lip_sx}
    \begin{split}
         \Big| \mathbb{S}_{X}(\mu_1)(x) - \mathbb{S}_{X}(\mu_2)(x) \Big| & = \frac{1}{1 + 2dr} \, \Bigg| r\sum_{j=1}^d (\mu_1 -\mu_2)(x+he_j) + r\sum_{j=1}^d (\mu_1 - \mu_2)(x-he_j)\Bigg| \\
         & \leq \frac{2dr}{1+2dr} \, \|\mu_1 -\mu_2\|_{\infty}.
    \end{split}
    \end{equation}
    Therefore,  $\mathbb{S}_{X}$ is a contraction for the $\|\cdot\|_{\infty}$ norm. As a consequence, it has a unique fixed point $Y$, which is then the unique solution to \eqref{eq:implicit}. Point (1) is proved.

    Let us prove point (2). Assume that $X\geq 0$. Since $\mathbb{S}_{X}$ is a contraction, we have that $Y = \lim_{n\rightarrow \infty} \mathbb{S}_X^n(\mu)$ for any $\mu \in \R(S)$. In particular, taking $\mu = X$, 
    \begin{equation*}
        Y = \lim_{n\rightarrow \infty} \mathbb{S}_X^n(X).
    \end{equation*}
    It is easy to verify that for any $\mu \in \R(S)$, if $\mu \geq 0$, then $\mathbb{S}_X(\mu) \geq 0$. Therefore, we deduce that $\mathbb{S}_X^n(X) \geq 0$ for any $n$ by induction, and therefore $Y \geq 0$. If we assume that $X\in \mathcal{P}(S)$, then for any $\mu \in \mathcal{P}(S)$, we can deduce that $\mathbb{S}_X(\mu) \in \mathcal{P}(S)$. This yields that $Y\in \mathcal{P}(S)$.
    
    Point (3) is proved similarly, assuming that $X$ is $L$-Lipschitz, observing that if $\mu$ is $L$-Lipschitz continuous, then $\mathbb{S}_X(\mu)$ is $L$-Lipschitz continuous.

    Let us prove the last statement. Taking any $i\in \{1,2,\ldots,d\}$, we define $\bar{\omega}, \, \omega \in \mathbb{R}(S)$ as follows:
    \begin{equation*}
        \bar{\omega} (x) = (\nabla_h Y)_i (x) ,\qquad \omega(x) = (\nabla_h X)_i (x) , \qquad \forall  x\in S.
    \end{equation*}
    Then $\bar{\omega}$ is the fixed point of $ \mathbb{S}_{\omega}$ (replace $X$ by $\omega$ in \eqref{eq:map}).
Let $\gamma = 2dr / (1+2dr)$. Using $\bar{\omega}=\lim_{n\rightarrow \infty} \mathbb{S}_{\omega}^n(\omega) $, we deduce from \eqref{eq:lip_sx} that
\begin{equation*}
\| \bar{\omega}- \omega \|_\infty
\leq \sum_{k=0}^\infty
\| \mathbb{S}_{\omega}^{k+1}(\omega)- \mathbb{S}_{\omega}^k(\omega)  \|_{\infty}
\leq \sum_{k=0}^\infty
\gamma^k \| \mathbb{S}_\omega(\omega) - \omega \|_{\infty}
= \frac{1}{1-\gamma} \| \mathbb{S}_{\omega}(\omega)-\omega \|_{\infty}.
\end{equation*}
It further follows that \begin{equation*}
    \begin{split}
        \|\bar{\omega} - \omega \|_{\infty} & \leq \frac{ \Delta t }{1-\gamma} \frac{c}{1+2dr}
        \big\| \Delta_h \omega  \big\|_{\infty} \\
        &\leq   \frac{\Delta t}{1-\gamma} \frac{c }{1+2dr} \max_{x\in S}
        \Bigg| \frac{\Delta_h X (x+he_i) - \Delta_h X (x - he_i)}{2h} \Bigg| \leq 2^{\alpha-1}cL' \Delta t h^{\alpha-1},
    \end{split}
    \end{equation*}
   where the last inequality is a consequence of the $\alpha$-H\"older continuity of $\Delta_h X$. Finally, let $ \bar{\omega}^{+}_i (x) = (\nabla_h^{+} Y)_i (x)$ and let $\omega^{+}_i(x) = (\nabla_h^{+} X)_i (x)$. By the same argument, we have that $ \|\bar{\omega}^{+}_i - \omega^{+}_i \|_{\infty} \leq 2^{\alpha-1}cL' (\Delta t h^{\alpha-1})$. Then, for any $x\in S$, it follows from the triangle inequality that
   \begin{equation*}
   \begin{split}
       \Big| \Delta_h Y(x) - \Delta_h X(x)\Big| = \Big| \sum_{i=1}^d \frac{\bar{\omega}^{+}_i (x) - \bar{\omega}^{+}_i (x-he_i) }{h}- \sum_{i=1}^d \frac{ {\omega}^{+}_i (x) -  {\omega}^{+}_i (x-he_i) }{h}\Big| = 2^{\alpha}d c L'(\Delta t h^{\alpha-2}).
   \end{split}
   \end{equation*}
The lemma is proved.
\end{proof}

\begin{proof}[Proof of Lemma \ref{lm:l}]
\label{proof:lm:l}
The differentiability of $H^c$ with respect to $p$ and the Lipschitz continuity of $H_p^c$ are proved in \cite[Thm.\@ 4.2.1, page 82]{JBHU}. For any $t_1,t_2 \in[0,1]$, we have
\begin{equation*}
\begin{split}
     H^c(t_1,x,p) - H^c(t_2,x,p) &
     = \sup_{v_1\in \mathbb{R}^d} \Big( \langle -p,v_1 \rangle - \ell^c(t_1,x,v_1) \Big)
     - \sup_{v_2\in \mathbb{R}^d}
\Big(     
     \langle -p,v_2 \rangle - \ell^c(t_1,x,v_2) \Big) \\
     & \leq \sup_{v\in \mathbb{R}^d} \Big( \ell^c(t_2,x,v)-  \ell^c(t_1,x,v) \Big) \\
     & \leq L^c_{\ell} \, |t_1-t_2|.
\end{split}
\end{equation*}
Using the relation of Fenchel, $-H_p^c(t,x,p) = \argmax_{v} \langle - p,v \rangle - \ell^c(t,x,v)$, and the continuous differentiability of $\ell^c$ , we have the first order optimality condition
\begin{equation*}
    p + \ell^c_v\Big(t,x, -H_p^c(t,x,p)\Big) = 0.
\end{equation*}
Fix $x\in \mathbb{T}^d$ and $p\in\mathbb{R}^d$. Take any $t_1$ and $t_2$ in $[0,1]$. By the above equation,
\begin{equation*}
    \ell^c_v\Big(t_1,x, -H_p^c(t_1,x,p)\Big)  = \ell^c_v\Big(t_2,x, -H_p^c(t_2,x,p)\Big).
\end{equation*}
The strong convexity of $\ell^c$ implies that
\begin{equation*}
\begin{split}
    & \alpha \big\| H_p^c(t_1,x,p) -H_p^c(t_2,x,p)   \big\|^2 \\
   & \qquad  \leq \big\langle \ell^c_v\big(t_1,x, -H_p^c(t_1,x,p)\big) - \ell^c_v\big(t_1,x, -H_p^c(t_2,x,p)\big) ,  H_p^c(t_2,x,p) - H_p^c(t_1,x,p) \big\rangle \\
   & \qquad =  \big\langle \ell^c_v\big(t_2,x, -H_p^c(t_2,x,p)\big) - \ell^c_v\big(t_1,x, -H_p^c(t_2,x,p)\big) ,  H_p^c(t_2,x,p) - H_p^c(t_1,x,p) \big\rangle \\
   & \qquad \leq L_{\ell}^c \big|t_1- t_2\big|  \big\| H_p^c(t_1,x,p) -H_p^c(t_2,x,p)   \big\|,
\end{split}
\end{equation*}
where the last inequality is a consequence of the Lipschitz continuity of $\ell_v^c$ with respect to $t$. The lemma is proved.
\end{proof}

\begin{proof}[Proof of Lemma \ref{lem:hD_diff}] \label{proof:hD_diff}
The first three claims can be shown with the same arguments as those of the proof of Lemma \ref{lm:l}. Since $H_p^D$ is $\frac{1}{\alpha}$-Lipschitz continuous with respect to $p$, it is enough to prove \eqref{eq:bound:Hp} for $p=0$. Let $v(t,x)= - H_p^D(t,x,0)$. Since $v(t,x)$ is optimal in \eqref{eq:HD}, with $p=0$, we deduce that $v(t,x)$ minimizes $\ell(t,x,\cdot)$ over the closed ball of radius $D$. Using the strong convexity of $\ell$, it follows that
\begin{equation*}
\ell(t,x,0) + \langle p_0, v(t,x) \rangle
+ \frac{\alpha}{2} \| v(t,x,0) \|^2
\leq
\ell(t,x,v(t,x))
\leq
\ell (t,x,0),
\end{equation*}
from which we deduce that $\| v(t,x) \| \leq \frac{2}{\alpha} \| p_0 \|$, by Cauchy-Schwarz inequality.
\end{proof}

\begin{proof}[Proof of Lemma \ref{lem:hD_quad}] \label{proof:hD_quad}
By Fenchel's relation \cite[Cor.\@ 1.4.4]{JBHU}, we know that
\begin{equation} \label{eq:pridual1}
H^D(t,x,\bar{p}) =
- \langle \bar{p},\bar{v} \rangle
- \ell^D(t,x,\bar{v})
\quad
\text{and}
\quad
- \bar{p} \in \partial_v \ell^D(t,x,\bar{v}).
\end{equation}
Using the strong convexity of $\ell^D$, 
we obtain that
\begin{equation} \label{eq:pridual2}
    \ell^D(t,x,v)  \geq \ell^D(t,x,\bar{v}) - \langle \bar{p}, v- \bar{v} \rangle + \frac{\alpha}{2} \| v - \bar{v} \|^2.
\end{equation}
Summing up \eqref{eq:pridual1} and \eqref{eq:pridual2}, we obtain the following inequality:
\begin{equation}\label{eq:pridual}
H^D(t,x,\bar{p}) + \ell^D(t,x,v) + \langle \bar{p}, v\rangle
\geq
\frac{\alpha}{2} \,
\big\| v - \bar{v} \big\|^2.
\end{equation}
Multiplying \eqref{eq:pridual} by $\bar{m}$, multiplying \eqref{eq:pridual1} by $m$, and taking the difference, we obtain the desired inequality.
\end{proof}

\begin{proof}[Proof of Theorem \ref{thm1}, second part] \label{proof:thm1}
We prove here the continuity of the mapping $\phi$.
Since $\phi$ is the composition of \eqref{eq:DP}, \eqref{eq:OS} and \eqref{eq:kolmogorov}, it suffices to show that these three mappings are continuous.

\smallskip
\noindent
\textbf{Step 1:} Continuity of \textbf{HJB}. Take any  $\mu_1$ and $\mu_2$ in $\mathcal{P}_{m_0}(\tilde{\mathcal{T}},S)$. Let $u_1=\textbf{HJB}(\mu_1)$ and $u_2=\textbf{HJB}(\mu_2)$.
By Assumption \ref{ass3}, we have that for any $x\in S$,
\begin{equation*}
\begin{split}
    |(u_1 -u_2)(t,x)| \leq{} & \sup_{\|\omega\|\leq D} \, \big|\tilde{\ell}_{\mu_1}(t,x,\omega) -\tilde{\ell}_{{\mu_2}}(t,x,\omega)  \big|  \Delta t + \Big| \sum_{y\in S} \pi(t,x,y,\omega ) \big(u_1(t+1,y)-u_2(t+1, y) \big)\Big| \\
     \leq{} & L_f'\|(\mu_1-\mu_2)(t,\cdot)\|_{2}\Delta t + \|(u_1-u_2)(t+1,\cdot)\|_{\infty},
\end{split}
\end{equation*}
where the last inequality follows from the Lipschitz continuity of $f$ and Assumption \ref{ass3}. 
{Since $\mu_1(t,\cdot), \mu_2(t,\cdot) \in \mathcal{P}(S)$ for any $t\in \tilde{\mathcal{T}}$, we have that $\mu_1(t,s), \mu_2(t,s) \in [0,1]$ for any $(t,s)\in \tilde{\mathcal{T}}\times S$, 
 which implies that  $ \|\mu_1-\mu_2\|_{\infty,\infty}\leq 1$.}
Combining this with the fact that $\|(u_2-u_1)(T,\cdot)\|_{\infty}=0$, it follows that
{\begin{equation} \label{eq:reg_u}
    \| u_1- u_2 \|_{\infty,\infty} \leq L_f'\Delta t \sum_{ t\in\mathcal{T}} \|(\mu_1-\mu_2)(t,\cdot)\|_{2} \leq L_f' \|\mu_1 - \mu_2\|_{\infty,2}\leq L_f' \|\mu_1 - \mu_2\|^{1/2}_{\infty,1}.
\end{equation}}

\smallskip
\noindent
\textbf{Step 2}: Continuity of \textbf{V}. Let $v_1=\textbf{V}(u_1)$ and ${v}_2= \textbf{V}(u_2)$. By the equivalent form of \eqref{eq:OS}, we have $v_1(t,x)= -H_p[p_{1,1}](t,x)$ and $v_2(t,x)= - H_p[p_{1,2}](t,x)$,
where
\begin{equation*}
p_{1,1}(t,x) = \sum_{s\in S} \pi_1(t,x,s)u_1(t+1,s)
\quad \text{and} \quad
p_{1,2}(t,x) = \sum_{s\in S} \pi_1(t,x,s)u_2(t+1,s).
\end{equation*}
By the $(1/\alpha)$-Lipschitz continuity of $H_p(t,x,p)$ on $p$, we have for any $(t,x)\in\mathcal{T}\times S$
\begin{equation*}
    \|v_1(t,x) -{v}_2(t,x)\| \leq \frac{1}{\alpha} \|p_{1,1}(t,x) - {p}_{1,2}(t,x)\|\leq \frac{1}{\alpha} \|\pi_1\|_{\infty,\infty,1} \|u_1- {u}_2 \|_{\infty,\infty},
\end{equation*}
 where $\|\pi_1\|_{\infty,\infty,1} = \max_{t,x} \sum_{s} \|\pi_1(t,x,s)\|$.

\smallskip
\noindent
\textbf{Step 3:} Continuity of \textbf{FP}. Let $m_1= \textbf{FP}(v_1)$ and $m_2= \textbf{FP}(v_2)$.
Then,
\begin{equation*}
 \begin{cases}
       (m_1-{m}_2)(t+1,y) = \sum_{x\in S} \pi[v_1](t,x,y) (m_1-{m}_2)(t,x) + \delta_{v_1,{v}_2,{m}_2}(t,y), \quad \\
       (m_1-{m}_2)(0,y) = 0,
 \end{cases}
 \end{equation*}
 where $\delta_{v_1,{v}_2,{m}_2}(t,y) = \Delta t \sum_{x\in S} \pi_1(t,x,y)(v_1-{v}_2)(t,x) {m}_2(t,x)$.
 Since $\pi[v_1]$ is a transition process, we have
 \begin{equation*}
     \| m_1(t+1,\cdot)- {m}_2(t+1,\cdot)\|_1 \leq  \|m_1(t,\cdot)- {m}_2(t,\cdot)\|_1 + \| \delta_{v_1,{v}_2,{m}_2}(t,\cdot)\|_1.
 \end{equation*}
 The second term $\| \delta_{v_1,{v}_2,{m}_2}(t,\cdot)\|_1$ is estimated with H\"older's inequality:
 \begin{equation*}
 \begin{split}
     \| \delta_{v_1,{v}_2,{m}_2}(t,\cdot)\|_1 & {} = \Delta t \Big|\sum_{y\in S, x\in S} \pi_1(t,x,y)(v_1 -{v}_2)(t,x) {m}_2(t,x) \Big| \\
     &\leq \Delta t\| {m}_2 (t,\cdot)\|_1 \| \pi_1\|_{\infty,\infty,1} \|v_1 -{v}_2 \|_{\infty,\infty} = \Delta t \| \pi_1\|_{\infty,\infty,1} \| v_1-{v}_2 \|_{\infty,\infty},
 \end{split}
 \end{equation*}
where the last equality is a consequence of $m_2 \in \mathcal{P}(\tilde{\mathcal{T}},S)$. Therefore, we have
 \begin{equation*}
     \|m_1- {m}_2 \|_{\infty,1}\leq  \|\pi_1\|_{\infty,\infty,1} \|v_1 -{v}_2 \|_{\infty,\infty}.
 \end{equation*}
 The continuity of $\phi$ follows.
\end{proof}

\section{On the regularity of the continuous MFG system} \label{Appendix:B}

Recall that $Q = [0,1]\times \mathbb{T}^d$. For any $R>0$, let $\mathbf{B}_R \coloneqq Q\times B(0,R)$, where $B(0,R)$ is the closed ball in $\mathbb{R}^d$ with center $0$ and radius $R$. Let us refer the reader fto \cite[pages 8 and 51]{krylov2008lectures} for the definitions of the Sobolev space $W^k_p(Q)$ and the anisotropic Sobolev space $W^{1,2}_p(Q)$. For any $\delta\in(0,1)$, we define the local H\"older space
\begin{equation*}
    \mathcal{C}^{\delta/2,\delta,\delta}_{\text{loc}}(Q\times \mathbb{R}^d) = \Big\{ w \in \mathcal{C}(Q\times \mathbb{R}^d) \; \Big \vert \; w\mid_{\mathbf{B}_R} \in \mathcal{C}^{\delta/2,\delta,\delta}(\mathbf{B}_R ),\; \text{for any }R>0 \Big\},
\end{equation*}
where $w\mid_{\mathbf{B}_R}$ is the restriction of $w$ in $\mathbf{B}_R$ and where $\mathcal{C}^{\delta/2,\delta,\delta}(\mathbf{B}_R )$ denotes the sets of functions from  $\mathbf{B}_R$ to $\R$ which are H\"older continuous with respect to their first (resp.\@ second and third) variable with exponent $\delta/2$ (resp.\@ $\delta$).

\begin{ass}\label{ass:regular}
There exist $C>0$ and $0<\bar{r}<1$ such that for all $ (t,x)\in Q$, for all $ v\in \mathbb{R}^d$ and for all $m\in \mathbb{L}^{\infty}(\mathbb{T}^d)$ satisfying $m \geq 0$ and $\int_{\mathbb{T}^d} m(x)dx =1$, it holds:
 \begin{itemize}
     \item $\ell^c(t,x,v)\leq C\|v\|^2 +C $ and $|f^c(t,x,m)|\leq C$;
     \item  $\ell^c$ and $\ell^c_v$ are continuously differentiable, and $\ell^c_{vx}, \ell^c_{vv} \in  \mathcal{C}^{\bar{r}/2,\bar{r},\bar{r}}_{\text{loc}}(Q\times \mathbb{R}^d)$;  
     \item $m_0^c \in \mathcal{C}^{2+ \bar{r}}(\mathbb{T}^d)$, and $  g^c \in \mathcal{C}^{3}(\mathbb{T}^d)$.
 \end{itemize}
\end{ass}

\begin{lem}\label{lm:hamiltonian}
Let Assumptions \ref{ass:continuous} and \ref{ass:regular} hold true. Then the Hamiltonian $H^c$ is continuously differentiable and $H_p^c$ is also continuously differentiable. Moreover, $H^c_{px} \in \mathcal{C}^{\bar{r}/2,\bar{r},\bar{r}}_{\textnormal{loc}}(Q\times \mathbb{R}^d)$ and $H^c_{pp} \in  \mathcal{C}^{\bar{r}/2,\bar{r},\bar{r}}_{\textnormal{loc}}(Q\times \mathbb{R}^d)$.
\end{lem}

\begin{proof}
Fix any $(t_0,x_0,p_0) \in (0,1)\times \mathbb{T}^d \times \mathbb{R}^d$. By the strong convexity of $\ell^c$ w.r.t.\@ $v$, there exists a unique $v_0\in \mathbb{R}^d$ such that
$H^c(t_0,x_0,p_0) = - \langle p_0, v_0\rangle - \ell^c(t_0,x_0,v_0)$.
The first order optimality condition writes
\begin{equation*}
    - p_0 - \ell^c_v(t_0,x_0,v_0) = 0.
\end{equation*}
Since $\ell^c$ is strongly convex, we have that $\ell^c_{vv}(t_0,x_0,v_0)$ is invertible. By the implicit function theorem, there exist a neighborhood $\mathcal{A}$ of $(t_0,x_0,p_0)$ and a function $v(t,x,p)$ from $\mathcal{A}$ to $\mathbb{R}^d$ such that for all $(t,x,p)\in \mathcal{A}$,
\begin{equation} \label{eq:first_order}
    - p - \ell^c_v (t,x, v(t,x,p)) = 0.
\end{equation}
Since $\ell_v^c$ is continuously differentiable, $v(t,x,p)$ is continuously differentiable. Moreover,
\begin{align*}
 v_x(t,x,p) & = \Big(\ell^c_{vv}(t,x,v(t,x,p))\Big)^{-1} \ell^c_{vx}(t,x,v(t,x,p)) ,\\
    v_p(t,x,p) & = \Big(\ell^c_{vv}(t,x,v(t,x,p))\Big)^{-1} .
\end{align*}
By the regularity of $\ell^c_{vv}$ and $\ell^c_{vx}$, we deduce that $v_x, v_p \in \mathcal{C}^{\bar{r}/2,\bar{r},\bar{r}}(\mathcal{A})$. The convexity of $\ell^c$ and the first order optimality condition \eqref{eq:first_order} imply that
\begin{equation*}
    H^c(t,x,p) = - \langle p, v(t,x,p)  \rangle - \ell^c (t,x, v(t,x,p)), \qquad \forall  (t,x,p)\in \mathcal{A}.
\end{equation*}
We deduce that  $H^c\mid_{\mathcal{A}} \in \mathcal{C}^1(\mathcal{A})$ by the regularity of $v$ and $\ell^c$. Differentiating the above equation with respect to $p$ and using \eqref{eq:first_order}, we obtain that $H_p^c(t,x,p) = - v(t,x,p)$, for all $(t,x,p) \in \mathcal{A}$.
Then, deriving $H^c_p$ with respect to $x$ and $p$, we obtain
\begin{equation*}
    H_{px}^c(t,x,p) = - v_x(t,x,p), \qquad H_{pp}^c(t,x,p) = -v_{p}(t,x,p), \qquad \forall  (t,x,p)\in \mathcal{A}.
\end{equation*}
The conclusion follows from the regularity of $v$, $v_x$ and $v_p$.
\end{proof}

\begin{thm}\label{thm:regular}
Under Assumptions \ref{ass:continuous} and \ref{ass:regular}, the continuous system \eqref{eq:mfg} has a unique solution $(u^{*}, v^{*}, m^{*})$ satisfying Assumption \ref{ass:sol+} for any $r<\bar{r}$.
\end{thm}

\begin{proof}
Fixing any $0<r<\bar{r}$, we will prove that Assumption \ref{ass:sol+} is satisfied for $r$. Under Assumptions \ref{ass:continuous} and \ref{ass:regular}, according to \cite[Thm.\@ 1]{bonnans2021schauder}, there exists $r' \in (0,\bar{r}]$ such that the continuous system \eqref{eq:mfg} has a unique classical solution $(u^{*}, v^{*}, m^{*})$ with 
\begin{equation}\label{eq:regular}
    u^{*}, \, m^{*}\in \mathcal{C}^{1+r'/2, 2+r'}(Q),
\quad    
    v^{*} \in \mathcal{C}^{r'}(Q), \quad \text{and} \quad \nabla v^{*} \in \mathcal{C}^{r'}(Q, \mathbb{R}^{d\times d}).
\end{equation}

\smallskip
\noindent
\textbf{Step 1:} Regularity of $\nabla u^{*}$. By \eqref{eq:regular}, we have that $u^{*}\in \mathcal{C}^1(Q)$. This implies that $\nabla u^{*} \in \mathbb{L}^{\infty}(Q)$.   Let $ u^{*}_{x_i}$ be the partial derivative of $u^{*}$ w.r.t.\@ $x_i$. Then, $u^{*}_{x_i} $ is a weak solution of the following linear equation:
\begin{equation*}
    \begin{cases}
    -\partial_t w (t,x) - \sigma \Delta w (t,x) + H^c_p[\nabla u^{*}] \nabla w(t,x) =  \check{f}_0(t,x), \qquad &  (t,x)\in Q \,,
   \\
    w (1,x) = g_{x_i}(x), & x\in \mathbb{T}^d,
    \end{cases}
\end{equation*}
where 
\begin{equation*}
   \check{f}_0(t,x) = D_{x_i} f^c(t,x,m^{*}(t)) - H^c_{x_i}[ \nabla u^{*}] (t,x),
\end{equation*}
where $D_{x_i} $ denotes the weak derivative w.r.t.\@ $x_i$. By the regularity of $H^c_p$, $H^c$, $f^c$, and $u^{*}$, we deduce that $ H^c_p[\nabla u^{*}] \in \mathbb{L}^{\infty}(Q) $ and $ \check{f}_0 \in \mathbb{L}^{\infty}(Q) $. Moreover, the regularity of $g$ implies that $g_{x_i} \in W^2_{\infty}(\mathbb{T}^d)$. Then \cite[Thm.\@ 4]{bonnans2021schauder}  shows that $u^{*}_{x_i}$ is the unique weak solution and $ u^{*}_{x_i} \in W^{1,2}_p (Q) \subset W^1_p(Q)$ for any $p>d+1$. Morrey's inequality \cite[Lem.\@ 4.28]{adams2003sobolev} implies that
\begin{equation*}
    u^{*}_{x_i} \in \mathcal{C}^{\gamma}(Q), \quad \text{ with } \ \gamma = 1 - \frac{d+1}{p}.
\end{equation*}
Taking $p = \frac{d+1}{1-r/\bar{r}}$, we have that $u^{*}_{x_i} \in \mathcal{C}^{r/\bar{r}}(Q)$. The same result follows for $\nabla u^{*}$.

\smallskip
\noindent
\textbf{Step 2:} Regularity of $u^{*}$.
Let $\varphi\in \mathcal{C}^{\infty}(\mathbb{R}^d)$ such that $\varphi (x) = 1$ for $x\in B(0,2\sqrt{d})$ and $\varphi (x) = 0$ for $x\notin \Omega \coloneqq B(0,3\sqrt{d})$. It is straightforward that $B(0,2\sqrt{d})$ contains a neighborhood of $\mathbb{T}^d$. Let us set $Q' = (0,1)\times \Omega$. 

Since $u^{*}$ can be identified to a periodic function over $\mathbb{R}^d$, we define $\Check{u} = u^{*} \varphi $. Then, $\check{u}$ is the solution of the following equation:
\begin{equation*}
    \begin{cases}
    -\partial_t \check{u}(t,x) - \sigma \Delta \check{u}(t,x) =  \check{f}_1(t,x), \qquad &  (t,x)\in Q' \,,
    \\
    \check{u}(t,x) = 0 , & (t,x)\in (0,1)\times \partial \Omega,
    \\
    \check{u}(1,x) = g(x)\varphi(x), & x\in \Omega,
    \end{cases}
\end{equation*}
where 
\begin{equation*}
    \check{f}_1(t,x) =  \varphi(x) \Big(f^c(t,x,m^{*}(t)) - H^c(t,x,\nabla u^{*}(t,x))\Big) - 2\sigma \langle \nabla \varphi(x) , \nabla u^{*}(t,x) \rangle- \sigma u^{*}(t,x) \Delta \varphi(x).
\end{equation*}
By the regularity of $f^c$ and $m^{*}$, we deduce the following: For any $(t_1,x_1), (t_2,x_2) \in Q'$,
\begin{equation*}
\begin{split}
     f^c(t_1,x_1,m^{*}(t_1)) - f^c(t_2,x_2,m^{*}(t_2)) &\leq L_f^c (|t_1-t_2| + \|x_1-x_2\|) + L_f^c \|m^{*}(t_1) - m^{*}(t_2)\|_{\mathbb{L}^2} \\
     & \leq L_f^c (|t_1-t_2| + \|x_1-x_2\|) + L_f^c \|m^{*}(t_1) - m^{*}(t_2)\|_{\mathbb{L}^\infty} \\
     & \leq C(|t_1-t_2|^{r/2} + \|x_1-x_2\|^r),
\end{split}
\end{equation*}
for some constant $C$. Using the regularity properties of $u^{*}$, $\nabla u^{*}$ and $H^c$, we obtain that $\check{f}_1 \in \mathcal{C}^{r/2,r}(Q')$. The final condition lies in $\mathcal{C}^r(\Omega)$ by Assumption \ref{ass:regular}. The boundary conditions satisfying the requirements in \cite[Thm.\@ 5.2]{ladyvzenskaja1988linear}, we deduce that $\check{u} \in \mathcal{C}^{1+r/2,2+r} (\bar{Q'})$, where $\bar{Q'}$ is the closure of $Q'$. By the definition of $\varphi$, we have that $u^{*}(t,x) = \check{u}(t,x)$ for all $(t,x)\in Q$. The regularity of $u^{*}$ follows.

\smallskip
\noindent
\textbf{Step 3:} Regularity of $v^{*}$. By \eqref{eq:mfg} (ii) and the regularity of $H^c$, we have
\begin{align*}
    v^{*}(t,x) & = - H^c_p(t,x,\nabla u^{*}(t,x));\\
    \nabla v^{*} (t,x) & = - H^c_{px} (t,x, \nabla u^{*}(t,x) ) - H^c_{pp} (t,x,\nabla u^{*}(t,x)) D_{xx} u^{*}(t,x).
\end{align*}
Since $H_p^c$ is continuously differentiable and $\nabla u^{*}\in \mathcal{C}^{r/\bar{r}}(Q)$, we deduce that $ v^{*} \in \mathcal{C}^{r/\bar{r}} (Q) \subset \mathcal{C}^{r} (Q) $. By a similar argument, from the regularity of $H^c_{px}, $ and $ H^c_{pp}$ in Lemma \ref{lm:hamiltonian} and the regularity of $ \nabla u^{*}$ and $D_{xx} u^{*}$, we have $\nabla v^{*} \in \mathcal{C}^{r/2,r}(Q)$.

\smallskip
\noindent
\textbf{Step 4:} Regularity of $m^{*}$. Since $m^{*}\in \mathcal{C}^{1+r'/2,2+r'}(Q)$ and $Q$ is bounded, we have $m^{*} \in W^{1,2}_p(Q)$ for any $p>d+2$. By \cite[Lem.\@ 3.3]{ladyvzenskaja1988linear}, it holds that
\begin{equation*}
    m^{*}\in \mathcal{C}^{\gamma/2, \gamma}(Q) \text{ and } \nabla m^{*}\in \mathcal{C}^{\gamma/2, \gamma}(Q), \text{ with } \gamma = 1-\frac{d+2}{p}.
\end{equation*}
Taking $p = \frac{d+2}{1-r}$, it follows that $m^{*} \in \mathcal{C}^{r/2,r}(Q)$ and $ \nabla m^{*} \in \mathcal{C}^{r/2,r}(Q)$.

Let us define $\check{m} = m^{*} \varphi$. Then $\check{m}$ satisfies the following equation:
\begin{equation*}
\begin{cases}
    \partial_t \check{m}(t,x) - \sigma \Delta \check{m}(t,x) + \langle v^{*}, \nabla \check{m} \rangle (t,x) + \text{div} (v^{*}) \check{m}(t,x) =  \check{f}_2(t,x), \qquad &  (t,x)\in Q' \,,
    \\
    \check{m}(t,x) = 0 , & (t,x)\in (0,1)\times \partial \Omega,
    \\
    \check{m}(0,x) = m_0(x)\varphi(x), & x\in \Omega,
\end{cases}
\end{equation*}
where 
\begin{equation*}
    \check{f}_2(t,x) =  - 2\sigma \langle  \nabla \varphi(x), \nabla m^{*}(t,x) \rangle - \sigma m^{*}(t,x) \Delta \varphi(x) + \langle v^{*}(t,x), \nabla\varphi(x) \rangle m^{*}(t,x).
\end{equation*}
From the regularity of $v^{*}, m^{*}$ and $\nabla m^{*}$, we deduce that $\check{f}_2 \in \mathcal{C}^{r/2,r}(Q')$. Combining with the regularity of $v^{*}$, $\nabla v^{*}$ and $m_0\varphi $, we deduce that $ \check{m}\in \mathcal{C}^{1+r/2,2+r}(\bar{Q'})$ by \cite[Thm.\@ 5.2]{ladyvzenskaja1988linear}. Therefore, $m^{*}\in \mathcal{C}^{1+r/2,2+r}(Q)$.
\end{proof}

\section{Construction of a numerical Hamiltonian}\label{Appendix:C}

This section, {as a complementary material to the rest of the article}, is dedicated to the construction of a numerical Hamiltonian satisfying the assumptions of \cite{achdou2016}, in a general framework (see equation \eqref{eq:numericalH}). Our main assumption is the strong convexity of the running cost with respect to the control variable.

Given a vector $q \in \mathbb{R}^{2d}$, we denote
\begin{equation*}
    {}^{\dag}q = (q_1,q_3,\ldots, q_{2d-1}), \quad  q^{\dag}= (q_2,q_4,\ldots, q_{2d}).
\end{equation*}
Following the terminology of \cite{achdou2016}, we call \emph{numerical Hamiltonian} a function $\mathcal{H}\colon [0,1]\times \mathbb{T}^d\times \mathbb{R}^{2d} \to \mathbb{R}$ satisfying the following conditions: For any $(t,x)\in [0,1]\times \mathbb{T}^d$,
\begin{itemize}
    \item[(g1)] [Monotonicity] $\mathcal{H}(t,x,\cdot)$ is nonincreasing w.r.t.\@ ${}^{\dag}q_i$ and nondecreasing w.r.t.\@ $q^{\dag}_i$ for all $i=1,2,\ldots,d$;
    \item[(g2)] [Consistency] For any $q$ such that ${}^{\dag}q= q^{\dag}$, it holds $\mathcal{H}(t,x,q) = H^c(t,x,q^{\dag})$;
    \item[(g3)] [Regularity] $\mathcal{H}(t,x,\cdot)$ is continuously differentiable;
    \item[(g4)] [Convexity] $\mathcal{H}(t,x,\cdot)$ is convex;
    \item[(g5)] There exists positive constants $c_1$, $c_2$, $c_3$ and $c_4$, independent of $(t,x)$, such that for any $q \in \R^{2d}$,
\begin{align}
\langle \mathcal{H}_q(t,x,q), q\rangle - \mathcal{H}(t,x,q) & \geq c_1 \|\mathcal{H}_q(t,x,q)\|^2 - c_2; \label{eq:E1} \\
       \|\mathcal{H}_q(t,x,q)\| & \leq c_3\|q\| + c_4. \label{eq:E2}
\end{align}
\end{itemize}
 
\begin{lem}\label{lm:numericalH}
Consider a function $\mathcal{H}\colon [0,1]\times \mathbb{T}^d\times \mathbb{R}^{2d} \to \mathbb{R}$ satisfying \textnormal{(g3)-(g4)}. Assume that $\mathcal{H}_q$ is uniformly Lipschitz continuous w.r.t.\@ $q$, $\mathcal{H}(t,x,0)$ is bounded from above, and $\mathcal{H}_q(t,x,0)$ is uniformly bounded. Then $\mathcal{H}$ satisfies \textnormal{(g5)}.
\end{lem}

\begin{proof}
 Let $L$ be the Lipschitz constant of $\mathcal{H}_q(t,x,\cdot)$. Applying inequality \eqref{eq:nesterov}, we obtain that
\begin{equation*}
     \frac{1}{2L} \| \mathcal{H}_q(t,x,0) - \mathcal{H}_q(t,x,q)\|^2 \leq \mathcal{H}(t,x,0) - \mathcal{H}(t,x,q) + \langle \mathcal{H}_q(t,x,q), q\rangle .
\end{equation*}
Applying inequality $\|a-b\|^2\geq \|b\|^2/2 - \|a\|^2 $, we deduce that
\begin{equation*}
     \langle \mathcal{H}_q(t,x,q), q\rangle -\mathcal{H}(t,x,q) \geq \frac{1}{4L}\|\mathcal{H}_q(t,x,q)\|^2 - \frac{1}{2L}\|\mathcal{H}_q(t,x,0)\|^2 -\mathcal{H}(t,x,0).
\end{equation*}
Since $\mathcal{H}(t,x,0)$ is bounded from above and since $\mathcal{H}_q(t,x,0)$ is uniformly bounded, \eqref{eq:E1} is satisfied. Inequality \eqref{eq:E2} is obvious by the uniform Lipschitz continuity of $\mathcal{H}_q$.
\end{proof}

Assume that the running cost $\ell^c(t,x,v)\colon [0,1]\times \mathbb{T}^d \times \mathbb{R}^d \rightarrow \mathbb{R}$ is uniformly $\alpha^c$-convex w.r.t.\@ $v$ with some $\alpha^c>0$. Then, $\ell^c$ can be decomposed as
\begin{equation*}
    \ell^c(t,x,v) = \ell_0^c(t,x,v) + \frac{\alpha^c}{2}\|v\|^2,
\end{equation*}
where $\ell^c_0$ is convex w.r.t.\@ $v$. We propose the following definition for a numerical Hamiltonian:
\begin{equation}\label{eq:numericalH}
    \mathcal{H}(t,x,q) = \sup_{
\begin{subarray}{c}    
    v \in \R^d, \,v \geq 0 \\ u \in \R^d, \, u \leq 0
\end{subarray}    
    } \,
\Big(  -\langle v,{}^{\dag}q \rangle -\langle u, q^{\dag} \rangle - \ell^c_0(t,x, v+u) -\frac{\alpha^c}{2} \big( \| v \|^2 + \|u\|^2 \big) \Big) .
\end{equation}

\begin{thm}\label{thm:numericalH}
  Assume that $\ell^c \colon [0,1]\times \mathbb{T}^d \times \mathbb{R}^d \rightarrow \mathbb{R}$ is $\alpha^c$-convex with respect to its third variable, $\ell^c$ is bounded from below by some constant $c$, and for some $v_0 \in\mathbb{R}^d$, there exists a constant $C(v_0)<+\infty$ such that for all $(t,x)\in [0,1]\times \mathbb{T}^d$, $\ell^c(t,x,v_0) \leq C(v_0)$. Then the function $\mathcal{H}$ defined by \eqref{eq:numericalH} is a numerical Hamiltonian, for the Hamiltonian $H^c$ defined by \eqref{eq:Hc}.
\end{thm}

\begin{proof}
The condition (g1) can be easily deduced from the nonnegativity and nonpositivity constraints for $v$ and $u$ in \eqref{eq:numericalH}. 

\smallskip
\noindent
\textbf{Step 1:} Proof of (g2). Let us take any $q\in\mathbb{R}^{2d}$, such that ${}^{\dag}q = q^{\dag}$. Then, we deduce that
\begin{equation*}
    \mathcal{H}(t,x,q) = \sup_{v\geq 0,\, u\leq 0}  -\langle v+u, q^{\dag} \rangle - \ell^c_0(t,x, v+u) -\frac{\alpha^c}{2} \big( \|v\|^2 + \|u\|^2 \big).
\end{equation*}
Since $v\geq0$ and $u\leq 0$, we have that $\|v\|^2+\|u\|^2\geq \|u+v\|^2$. Then, \begin{equation*}
    \mathcal{H}(t,x,q) \leq \sup_{v\geq 0,\, u\leq 0}  -\langle v+u, q^{\dag} \rangle - \ell^c_0(t,x, v+u) -\frac{\alpha^c}{2}(\|v+u\|^2) =  H^c(t,x,q^{\dag}).
\end{equation*}
Conversely, take $v^{*} =- H^c_p(t,x,q^{\dag})$, $v^{*}_{+} = \{\max\{0,v_i^{*}\}\}_{i=1,\ldots d}$ and $v^{*}_{-}= \{\min\{0,v_i^{*}\}\}_{i=1,\ldots d}$. We have that $v^{*}_{+}\geq 0$, $v^{*}_{-}\leq 0$, $v^{*} = v^{*}_{+} + v^{*}_{-}$ and $\|v^{*}\|^{2} =\|v^{*}_{+}\|^2 + \|v^{*}_{-}\|^2 $. Thus, by Fenchel's relation, it follows that
\begin{equation*}
    H^c(t,x,q^{\dag}) = -\langle v^{*}_{+} + v^{*}_{-}, q^{\dag} \rangle - \ell^c_0(t,x, v^{*}_{+} + v^{*}_{-}) -\frac{\alpha^c}{2} \big( \|v^{*}_{+}\|^2 +\| v^{*}_{-}\|^2 \big) \leq \mathcal{H}(t,x,q).
\end{equation*}

\smallskip
\noindent
\textbf{Step 2:} Proof of (g3)-(g4). Consider the function $\bar{\ell}^c \colon  [0,1]\times \mathbb{T}^d\times \mathbb{R}^{2d}\to \mathbb{R}$ defined by
\begin{equation*}
    \bar{\ell}^c(t,x,w) = \ell^c_0(t,x,{}^{\dag}w + w^{\dag}) + \frac{\alpha^c}{2}(\|{}^{\dag}w\|^2 + \|w^{\dag}\|^2) + \chi^{+}({}^{\dag}w) +\chi^{-}(w^{\dag}), 
\end{equation*}
where $\chi^{+}(x) = 0$ (resp.\@ $\chi^{-}(x)=0$) if $x\geq 0$ (resp.\@ $x\leq 0$) and infinity otherwise. It is obvious that $\bar{\ell}^c$ is uniformly $\alpha^c$-convex w.r.t.\@ $w$. The definition \eqref{eq:numericalH} implies that
\begin{equation*}
    \mathcal{H}(t,x,q) = (\bar{\ell}^{c})^{*}(t,x,-q).
\end{equation*}
By \cite[Thm.\@ 4.2.1]{JBHU}, $\mathcal{H}$ is convex and continuously differentiable w.r.t.\@ $q$ and $\mathcal{H}_q$ is uniformly $1/\alpha^c$-Lipschitz w.r.t.\@ $q$. 

\smallskip
\noindent
\textbf{Step 3:} Proof of (g5). We apply Lemma \ref{lm:numericalH} for the proof. Taking $q=0$, by the consistency of $\mathcal{H}$, we have for any $(t,x)\in [0,1]\times\mathbb{T}^d$ that 
\begin{equation*}
    \mathcal{H}(t,x,0) = H^c(t,x,0) =  - \inf_{v\in\mathbb{R}^d} \ \Big( \ell^c_0(t,x,v) + \frac{\alpha^c}{2}\|v\|^2 \Big) \leq -c,
\end{equation*}
By Fenchel's relation, it follows that
\begin{equation*}
     - \mathcal{H}_q(t,x,0) = \argmin_{v\geq 0, u\leq 0} \ \Big( \ell^c_0(t,x,v+u) + \frac{\alpha^c}{2}(\|v\|^2+\|u\|^2) \Big).
\end{equation*}
Let us set $v^{*}(t,x) = \argmin_{v\in\mathbb{R}^d} \ell^{c}(t,x,v)$.
By a similar argument to the one of Step 1, we have that ${}^{\dag}\mathcal{H}_q(t,x,0) = - v^{*}(t,x)_{+} $ and $\mathcal{H}_q(t,x,0)^{\dag} =  - v^{*}(t,x)_{-}$. In order to prove that $\mathcal{H}_q(t,x,0)$ is uniformly bounded, it suffices to show the boundedness of $v^{*}(t,x)$. By the strong convexity and boundedness assumptions of $\ell^c$, we deduce that for any $(t,x)\in [0,1]\times\mathbb{T}^d$,
\begin{equation*}
    C(v_0)\geq \ell^c(t,x,v_0) \geq \ell^c(t,x,v^{*}(t,x)) + \frac{\alpha^c}{2}\|v^{*}(t,x)-v_0\|^2 \geq c + \frac{\alpha^c}{2}\|v^{*}(t,x)-v_0\|^2.
\end{equation*}
This implies that $\|v^{*}\|_{\infty}\leq \|v_0\| + \sqrt{2(C(v_0)-c)/\alpha^c}$. The conclusion follows.
\end{proof}

\bibliographystyle{plain} 
\bibliography{biblio.bib} 


\end{document}